\documentclass[11pt]{amsart}
\usepackage{amsmath, amssymb}
\usepackage{amsfonts}
\usepackage{mathrsfs}
\usepackage[arrow,matrix,curve,cmtip,ps]{xy}

\usepackage{amsthm}
\usepackage{amsfonts}
\usepackage{amssymb}
\usepackage{graphicx}
\usepackage[arrow,matrix,curve,cmtip,ps]{xy}
\usepackage{subfigure}
\usepackage{pb-diagram}
 \usepackage{pb-xy}
\usepackage{caption}
\usepackage{amsthm}
\usepackage{amsmath, amssymb}

 \usepackage{amsmath,amscd}
\usepackage{amsfonts}
\usepackage{mathrsfs}
\usepackage[arrow,matrix,curve,cmtip,ps]{xy}
\usepackage{pb-diagram}
 \usepackage{pb-xy}

\usepackage{graphicx}

\usepackage{amsthm}

\usepackage{color}

\usepackage{xcolor}

\usepackage{setspace}

\allowdisplaybreaks

\newtheorem*{rep@theorem}{\rep@title}
\newcommand{\newreptheorem}[2]{%
\newenvironment{rep#1}[1]{%
 \def\rep@title{#2 \ref{##1}}%
 \begin{rep@theorem}}%
 {\end{rep@theorem}}}
\makeatother

\newreptheorem{theorem}{Theorem}
\newreptheorem{lemma}{Lemma}
\newreptheorem{proposition}{Proposition}
\newreptheorem{corollary}{Corollary}

\newtheorem{theorem}{Theorem}[section]
\newtheorem{lemma}[theorem]{Lemma}
\newtheorem{proposition}[theorem]{Proposition}
\newtheorem{corollary}[theorem]{Corollary}

\newtheorem*{theorem*}{Theorem}
\theoremstyle{remark}
\newtheorem{remark}[theorem]{Remark}
\newtheorem{definition}[theorem]{Definition}
\newtheorem{exam}[theorem]{Example}

\numberwithin{equation}{section}

\begin{document}
\title[($n$)-solvability and Milnor's invariants]{The ($n$)-solvable filtration of the link concordance group and Milnor's invariants}

\author{Carolyn Otto}

\address{Department of Mathematics \\ University of Wisconsin-Eau Claire}
\email{ottoa@uwec.edu}


\date{\today}



\begin{abstract}

We establish several new results about both the ($n$)-solvable filtration, $\{\mathcal{F}_n^m\}$, of the set of link concordance classes and the ($n$)-solvable filtration of the string link concordance group.  We first establish a relationship between Milnor's invariants and links, $L$, with certain restrictions on the 4-manifold bounded by $M_L$.  Using this relationship, we can relate ($n$)-solvability of a link (or string link) with its Milnor's $\overline{\mu}$-invariants. Specifically, we show that if a link is ($n$)-solvable, then its Milnor's invariants vanish for lengths up to $2^{n+2}-1$. Previously, there were no known results about the ``other half" of the filtration, namely $\mathcal{F}_{n.5}^m/\mathcal{F}_{n+1}^m$.   We establish the effect of the Bing doubling operator on ($n$)-solvability and using this, we show that  $\mathcal{F}_{n.5}^m/\mathcal{F}_{n+1}^m$ is nontrivial for links (and string links) with sufficiently many components.  Moreover, we show that these quotients contain an infinite cyclic subgroup.  We also show that links and string links modulo (1)-solvability is a nonabelian group.  We show that we can relate other filtrations with Milnor's invariants.  We show that if a link is $n$-positive, then its Milnor's invariants will also vanish for lengths up to $2^{n+2}-1$. Lastly, we prove that the Grope filtration, $\{\mathcal{G}_n^m\}$, of the set of link concordance classes is not the same as the ($n$)-solvable filtration.
\end{abstract}

\maketitle

\section{Introduction}
Much work has been done in the quest of understanding the ($n$)-solvable filtration.  In particular, many have studied successive quotients of this filtration and some of their contributions can be found in~\cite{Cha4},~\cite{CH2},~\cite{CHL3}, and ~\cite{Ha2}.

For example, Harvey first showed that $\mathcal{F}_n^m/\mathcal{F}^m_{n+1}$ is a nontrivial group that contains an infinitely generated subgroup~\cite{Ha2}.  She also showed that this subgroup is generated by boundary links (links with components that bound disjoint Seifert surfaces).  Cochran and Harvey generalized this result by showing that $\mathcal{F}^m_n/\mathcal{F}^m_{n.5}$ contains an infinitely generated subgroup~\cite{CH2}.  Again, this subgroup consists entirely of boundary links.

Up to this point, little has been known about the relationship of Milnor's invariants and ($n$)-solvability. We first establish the following relationship, which is the main theorem of this paper.

\begin{reptheorem}{theorem:newmain}Suppose $L$ is a link whose zero-framed surgery, $M_L$ bounds an orientable and compact 4-manifold $W$ that satisfies
\begin{itemize}
 \item[1.] $H_1(M_L) \rightarrow H_1(W; \mathbb{Z})$ is an isomorphism induced by the inclusion map;
  \item[2.]$H_2(W;\mathbb{Z})$ has a basis consisting of connected compact oriented surfaces $\{L_i\}$ with $\pi_1(L_i) \subset \pi_1(W)^{(n)}$.
   \end{itemize}
   Then $\bar \mu_L (I) = 0$ where $|I| \leq 2^{n+2}-1$.
\end{reptheorem}
Using this theorem we obtain the following relationship.
\begin{repcorollary}{theorem:main}If $L$ is an $(n)$-solvable link (or string link), then $\bar \mu_L(I)=0$ for $|I| \leq 2^{n+2}-1$.
\end{repcorollary}

In other words, if a link (or string link) is ($n$)-solvable, then all of its $\bar \mu$-invariants will vanish for lengths less than or equal to $2^{n+2}-1$.
Moreover, this theorem is sharp in the sense that we exhibit ($n$)-solvable links with $\bar \mu(I) \not = 0$ for $|I| = 2^{n+2}$.

We can also obtain a relationship between Milnor's invariants and other filtrations.  Specifically, we have the following result that relates Milnor's invariants to $n$-positive, $n$-negative and $n$-bipolar filtrations.  The definitions of these filtrations can be found in ~\cite{CHH}.

\begin{repcorollary}{cor:newresult}
If $L$ is an $n$-positive, $n$-negative, or $n$-bipolar link, then $\bar \mu_L(I)=0$ for $|I| \leq 2^{n+2}-1$.
\end{repcorollary}

We study the effects of Bing doubling on ($n$)-solvability. We show that solvability is not only preserved under this operator, but it increases the solvability by one.

\begin{repproposition}{proposition:bingdouble}If $L$ is an $(n)$-solvable link, then $BD(L)$ is $(n+1)$-solvable. Moreover,
if $L$ is an $(n.5)$-solvable link, then $BD(L)$ is $((n+1).5)$-solvable.
\end{repproposition}

Until this point, nothing was known about the ``other half" of the filtration, $\mathcal{F}_{n.5}^m/\mathcal{F}_{n+1}^m$. Using the above results, we show that the ``other half" of the ($n$)-solvable filtration is nontrivial.

\begin{reptheorem}{theorem:nontrivial}$\mathcal{F}^m_{n.5}/\mathcal{F}^m_{n+1}$ contains an infinite cyclic subgroup for $m \geq 3*2^{n+1}$.
\end{reptheorem}

The examples used come from iterated Bing doubles of links with nonvanishing $\bar \mu$-invariants. Thus, our examples are not concordant to boundary links, so the subgroups that they will generate will be different than those previously detected.  The result of Theorem~\ref{theorem:nontrivial} is still unknown for knots.

Since the knot concordance group, $\mathcal{C}$, is abelian, all successive quotients of the ($n$)-solvable filtration are abelian.  However, it is known that $\mathcal{C}^m$ is nonabelian for $m \geq 2$~\cite{LD}.  We have have shown that certain successive quotients are not abelian.

\begin{reptheorem}{theorem:abelian}$\mathcal{F}^m_{-0.5}/\mathcal{F}^m_1$ is nonabelian for $m \geq 3$.

\end{reptheorem}

Similar to the relationship between ($n$)-solvability and $\bar \mu$-invariants, we establish a relationship between $\bar \mu$-invariants and a link in which all of its components bound disjoint gropes of height $n$.  This relationship says that if all components of a link bound disjoint gropes of a certain height, then its $\bar \mu$ invariants vanish for certain lengths.  See Section 6 for the definition of a grope.

\begin{repcorollary}{cor:grope} A link $L$ with components that bound disjoint Gropes of height $n$ has $\bar \mu_L(I) = 0$ for $|I|\leq 2^{n}$.
\end{repcorollary}

A result of Lin~\cite{Lin} states that $k$-cobordant links will have the same $\bar \mu$-invariants up to length $2k$. Using this result, the proof of this proposition relies on showing that $L$ is $2^{n+1}$-cobordant to a slice link.

These two filtrations are related by the fact that $\mathcal{G}^m_{n+2} \subseteq \mathcal{F}^m_n$ for all $n \in \mathbb{N}$ and $m \geq 1$~\cite{COT}.  A natural question is whether or not these filtrations are actually the same.  We show that these filtrations differ at each stage.

\begin{repcorollary}{cor:main}$\mathcal{F}^m_n/\mathcal{G}^m_{n+2}$ is nontrivial for $m \geq 2^{n+2}$.  Moreover, $\mathbb{Z} \subset \mathcal{F}^m_n/\mathcal{G}^m_{n+2}$.
\end{repcorollary}

It is still unknown whether the previous result holds for knots.
\section{Preliminaries}
A \textbf{knot} is an embedding $S^1 \hookrightarrow S^3$.  The set of knots modulo concordance forms a group under the operation of connected sum, known as the knot concordance group $\mathcal{C}$. Two knots, $K$ and $J$ are said to be \emph{concordant} if $K \times \{0\}$ and $J \times \{1\}$ cobound a smoothly embedded annulus in $S^3 \times [0,1]$.

An \textbf{$m$-component link} is an embedding $\coprod_m S^1 \hookrightarrow S^3$. The connected sum operation is not well defined for links.  Therefore, in order to define a group structure on links, it is necessary to study string links.
\begin{definition}
Let $D$ be the unit disk, $I$ the unit interval and $\{p_1, p_2, \ldots, p_k\}$ be $k$ points in the interior of $D$.  A \textbf{$k$-component (pure) string link} is a smooth proper embedding $\sigma : \coprod_{i=1}^k I_i \rightarrow D \times I$ such that
\begin{align*}
\sigma |_{I_i}(0)&=\{p_i\} \times \{0\};\\
\sigma |_{I_i}(1)&=\{p_i\} \times \{1\}.
\end{align*}
The image of $I_i$ is called the $i^{th}$ \emph{string} of the string link.  An orientation on $\sigma$ is induced by the orientation of $I$.  Two string links $\sigma$ and $\sigma'$ are said to be \emph{equivalent} if there is an orientation preserving homeomorphism $h:D^2 \times I \rightarrow D^2 \times I$ such that $h$ fixes the boundary piecewise and $h(\sigma)=\sigma'$.
\end{definition}

The operation on string links is the stacking operation seen in the braid group.  If $L_1$ and $L_2$ are in $\mathcal{C}^m$, then $L_1L_2$ is the string link obtained by stacking $L_1$ on top of $L_2$.

The notion of concordance can be generalized for string links, see Figure~\ref{fig:stringlinkconcordance}.
\begin{definition}
Two $m$-component string links, $\sigma_1, \sigma_2$ are \textbf{concordant} if there exists a smooth embedding
$H:\coprod_m (I \times I) \rightarrow B^3 \times I$ that is transverse to the boundary and such that
$H|_{\coprod_m I\times \{0\}}=\sigma_1$, $H|_{\coprod_m I\times \{1\}}=\sigma_2$, and $H|_{\coprod_m \partial I\times I}=j_0 \times id_I$ where $j_0: \coprod_m \partial I \rightarrow S^2$.
\end{definition}

\begin{figure}[h!]
  \centering
 \includegraphics[height=4cm]{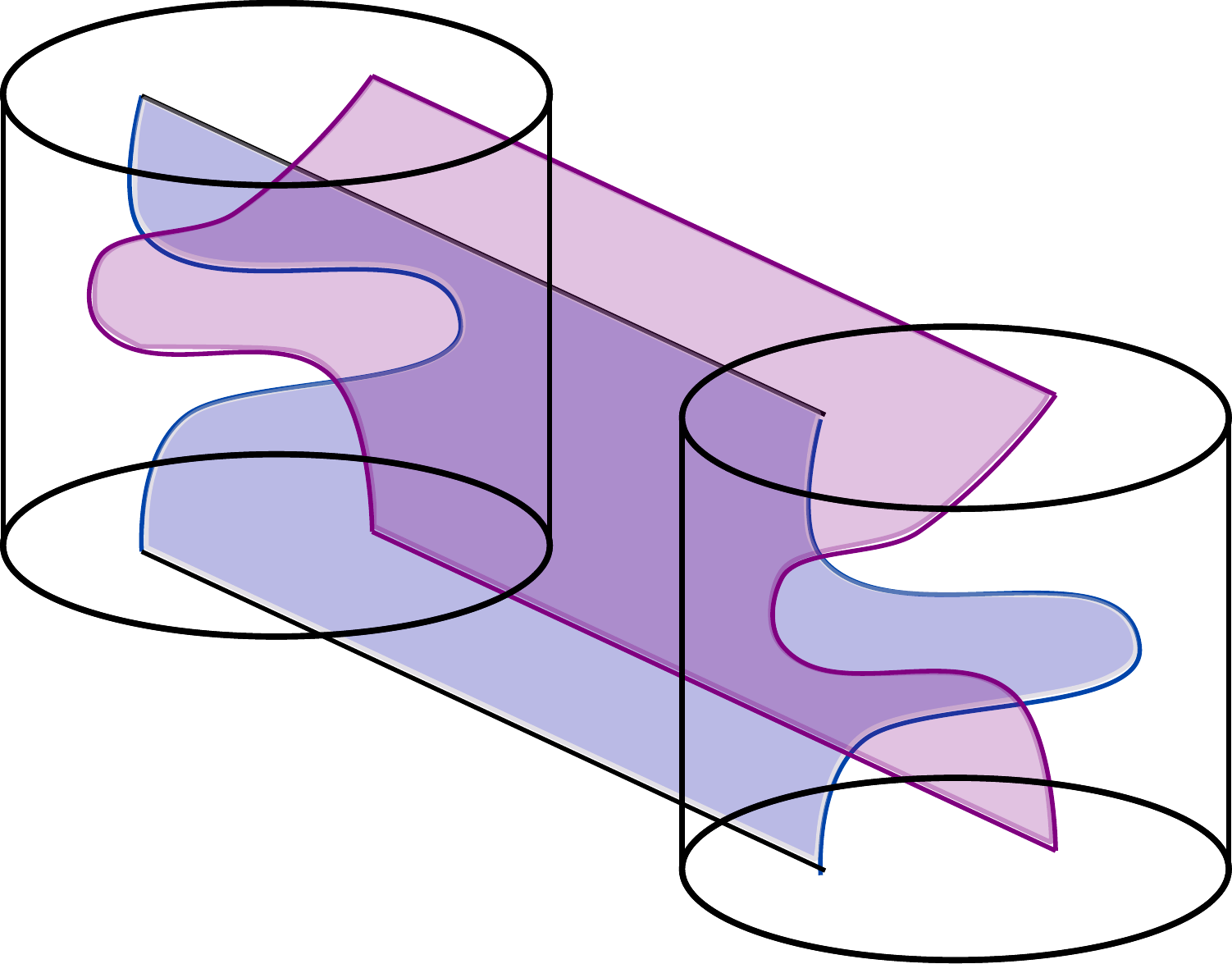}
 \caption[String link concordance]%
 {String link concordance}
 \label{fig:stringlinkconcordance}
\end{figure}

Under the operation of stacking, the concordance classes of $m$-component string links form a group, denoted $\mathcal{C}^m$, and is known as the string link concordance group.  The identity class of this group is the class of slice string links. The inverses are the string links obtained by reflecting the string link about $D \times \{1/2\}$ and reversing the orientation.  When $m=1$, $\mathcal{C}^m$ is the knot concordance group.  For $m \geq 2$, it has been shown that $\mathcal{C}^m$ is not abelian~\cite{LD}.

If $L$ is a string link, the closure of $L$, denoted $\hat L$, is the ordered, oriented link in $S^3$ obtained by gluing $\partial (D^2\times I)$ to $\partial (D^2 \times I)$ of the standard trivial string link using the identity map, see Figure~\ref{fig:closure}.  This gives a canonical way to obtain a link from a string link.  If two string links are concordant, then their closures are concordant as links.

\begin{figure}[h!]
 \centering
\subfigure[An example of a three component string link]{\includegraphics[height=4cm]{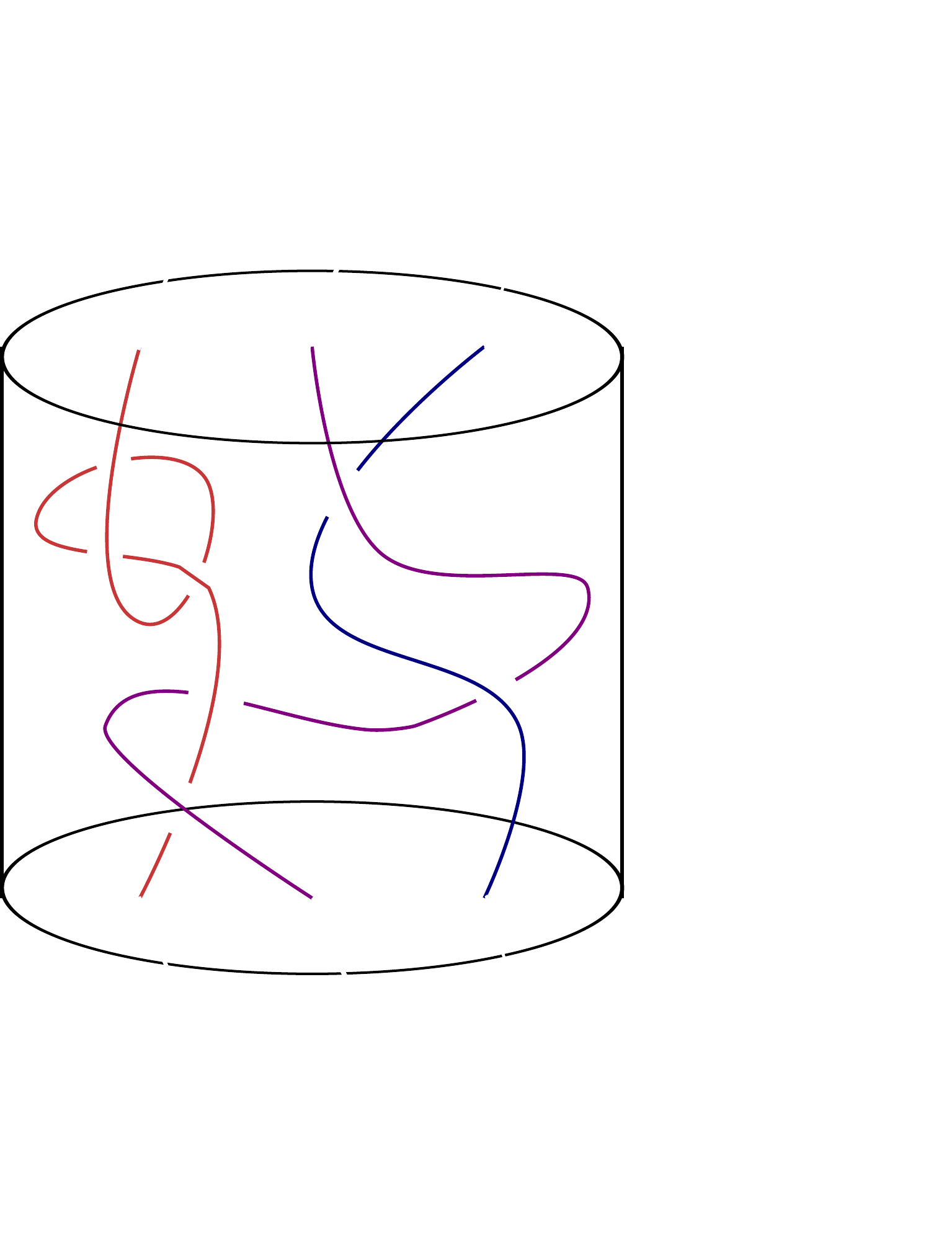}
\label{fig:stringlink}}\qquad \qquad
\subfigure[The closure of a string link]{\includegraphics[height=4cm]{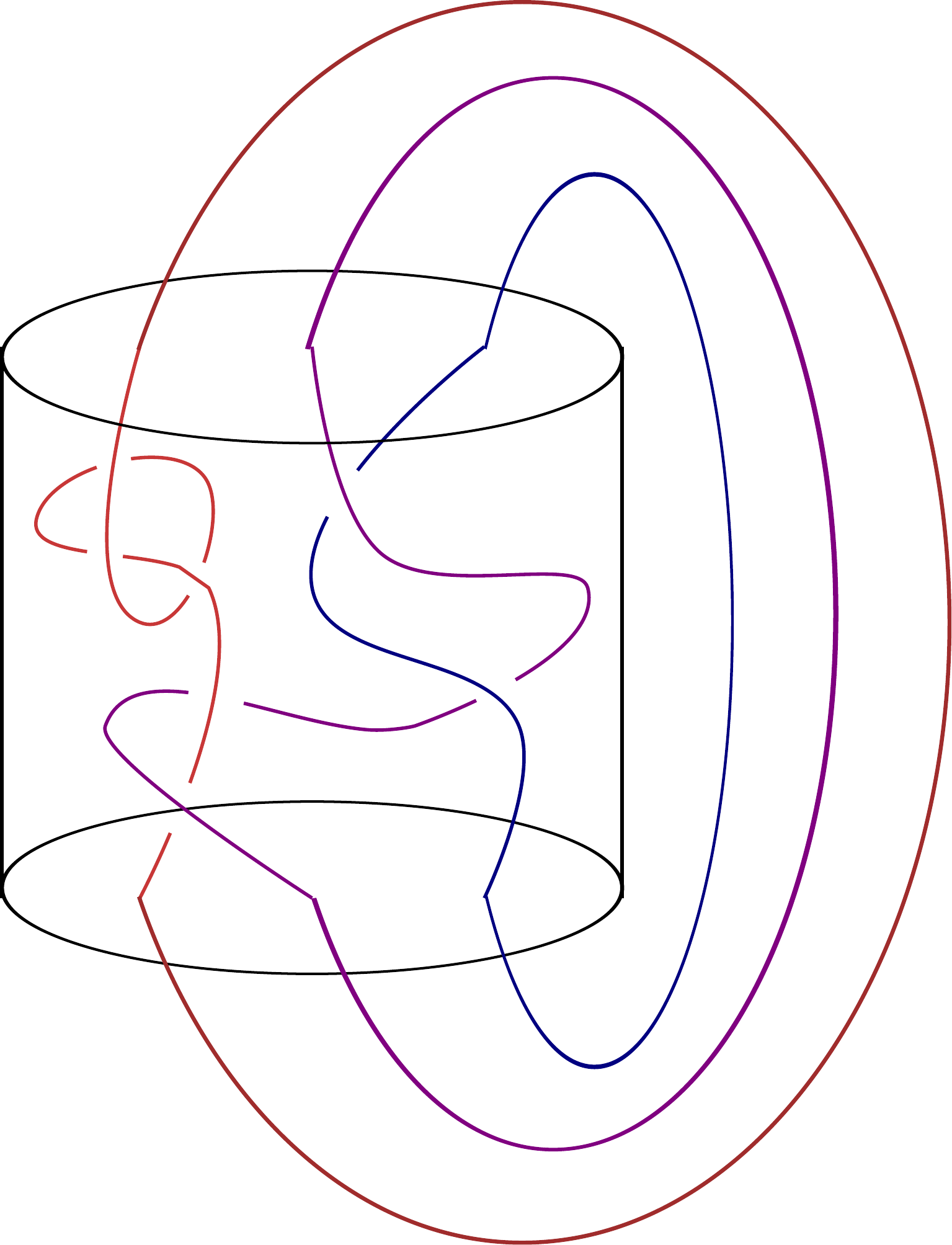}
\label{fig:closure}}
\caption[Example of a string link]%
{A string link and its closure}
\end{figure}

Every link has a string link representative. In other words, given any link, $L$ in $S^3$, there exists a string link $\sigma$ such that $\hat \sigma$ is isotopic to $L$.~\cite{HL1}

In order to study the structure of $\mathcal{C}^m$, Cochran, Orr and Teichner~\cite{COT} defined the ($n$)-solvable filtration, $\{\mathcal{F}^m_n\}$,
$$\{0\} \subset \cdots \subset \mathcal{F}^m_{n+1} \subset \mathcal{F}^m_{n.5} \subset \mathcal{F}^m_n \subset \cdots \subset \mathcal{F}^m_{0.5} \subset \mathcal{F}^m_0 \subset \mathcal{C}^m.$$

\begin{definition}
An $m$-component link $L$ is \textbf{($n$)-solvable} if the zero-framed surgery, $M_L$, bounds a compact, smooth 4-manifold, $W^4$, such that the following conditions hold:

i)  $H_1(M_L; \mathbb{Z}) \cong \mathbb{Z}^m$ and $H_1(M_L) \rightarrow H_1(W; \mathbb{Z})$ is an isomorphism induced by the inclusion map;

ii) $H_2(W; \mathbb{Z})$ has a basis consisting of connected, compact, oriented surfaces, $\{L_i, D_i\}_{i=1}^r$, embedded in $W$ with trivial normal bundles, where the surfaces are pairwise disjoint except that, for each $i$, $L_i$ intersects $D_i$ transversely once with positive sign;

iii) For all $i\text{, } \pi_1(L_i) \subset \pi_1(W)^{(n)}$ and  $\pi_1(D_i) \subset \pi_1(W)^{(n)}$ where $\pi_1(W)^{(n)}$ is the $n^{th}$ term of the derived series.  The derived series of a group $G$, denoted $G^{(n)}$ is defined recursively by $G^{(0)}:=G$ and $G^{(i)}:=[G^{(i-1)},G^{(i-1)}]$.

\end{definition}

The manifold $W$ is called an \textbf{($n$)-solution} for $L$ and a string link is \textbf{($n$)-solvable} if its closure in $S^3$ is an ($n$)-solvable link.

A link is \textbf{($n.5$)-solvable} if, in addition to the above, $ \pi_1(L_i) \subset \pi_1(W)^{(n+1)}$ for all $i$.  In this instance, $W$ is called an \textbf{($n.5$)-solution} for $L$.

We define the ($n$)-solvable filtration of the string link concordance group for $m \geq 1$ by setting $\mathcal{F}_n^m$ to be the set of ($n$)-solvable string links for $n \in \frac{1}{2}\mathbb{N}_0$.  It is known that $\mathcal{F}_n^m$ is a normal subgroup of $\mathcal{C}^m$ for all $m \geq 1$ and $n \in \frac{1}{2}\mathbb{N}_0$.  For convenience, $\mathcal{F}^m_{-0.5}$ will denoted the set of string links with all pairwise linking numbers equal to zero.

In the early 1950's, John Milnor defined a family of higher order linking numbers known as $\bar \mu$-invariants for oriented, ordered links in $S^3$~\cite{M1},~\cite{M2}.  These numbers are not link invariants in the typical sense since there is some indetermincy due to the choice of meridians of a link;  however, as invariants of string links they are well defined~\cite{HL1}.  In general, Milnor's invariants determine how deep the longitudes of each component lie in the lower central series of the link group. We will show in Theorem~\ref{theorem:main} that these invariants give information about the solvable filtration.  We will use this relationship to prove Theorem~\ref{theorem:nontrivial}.

Suppose $L$ is an $m$-component link in $S^3$.  Let $G=\pi_1(S^3-L)$ be the fundamental group of the complement of $L$ in $S^3$.  The \emph{lower central series} of $G$, denoted $G_i$ is recursively defined by $G_1:=G$ and $G_i:=[G_{i-1},G]$.  It is worthy to note that the derived series and the lower central series of a group $G$ are related by $G^{(n)} \subset G_{2^n}$.  Since $[G_r, G_s] \subseteq G_{r+s}$, it is a straight forward computation to achieve the relation.

Consider the nilpotent quotient group $G/G_k$.  A presentation of this group, given by Milnor~\cite{M2}, can be written
\begin{equation}G/G_k \cong \langle \alpha_1, \alpha_2, \dots, \alpha_m | [\alpha_1, l_1],[\alpha_2, l_2], \dots [\alpha_m, l_m], F_k \rangle \end{equation}
where $\alpha_1, \dots, \alpha_m$ are a choice of $m$ meridians for $L$, $F_k$ is the $k^{th}$ term of the lower central series of $F=F\langle \alpha_1, \dots, \alpha_m \rangle$, the free group on $m$ generators and $l_i$ is the $i^{th}$ longitude of L written as a product of the $\alpha_i$'s.  

 Let $\mathbb{Z}[[X_1, \dots, X_m]]$ be the ring of power series in $m$ noncommuting variables.  The \emph{Magnus embedding} is a map $E:\mathbb{Z}F \rightarrow \mathbb{Z}[[X_1, \dots, X_m]]$ defined by sending $\alpha_i \mapsto 1+X_i$ and $\alpha_i^{-1} \mapsto 1-X_i+X_i^2-X_i^3+\cdots$ for $1 \le i \le m$. Let $I=i_1i_2\dots i_{k-1}i_k$ be a string of integers amongst $\{1,\dots, m\}$ with possible repeats. The Magnus embedding of the longitude $l_{i_k}$ written as an element of $F$ (modulo $G_k$) has the form
 $$ E(l_{i_k})=1+\Sigma \mu_L(i_1\dots i_{k-1})X_{i_1}\cdots X_{i_{k-1}}.$$
 Milnor's invariant $\bar \mu_L (I)$ is defined as the residue class of $\mu_L(I)$ modulo the greatest common divisor of $\mu_L(\tilde I)$ where $\tilde I$ is any string of integers obtained from $I$ by deleting at least one integer (excluding $i_k$) and cyclically permuting the rest.  We refer to $|I|$ as the length of the Milnor's invariant. It is useful to note that the first nonvanishing $\bar \mu$-invariant, $\bar \mu_L(I)$ will be $\mu_L(I)$ since it is well-defined.

For $\bar \mu$-invariants of length two, the calculation measures the linking between two components, i.e. $\bar \mu_L(ij)$ is the linking number between the $i^{th}$ and $j^{th}$ components of $L$.  It is also known that $\bar \mu$-invariants are concordance invariants~\cite{CA}.

The following is a classical and well-known result of Milnor~\cite{M2}.

\begin{theorem}[Milnor]
The longitudes of $L$ lie in $G_{k-1}$ if and only if $F/F_k \cong G/G_k$.  In other words, $\bar \mu_L(I) = 0$ for $|I| \leq k-1$ if and only if $F/F_k \cong G/G_k$.
\label{theorem:milnor}
\end{theorem}

The following corollary allows us to detect whether certain Milnor's invariants are zero using the fundamental group of $M_L$, the zero-framed surgery on $L$.

\begin{corollary}
$F/F_{k+1} \cong G/G_{k+1}$ if and only if $F/F_{k} \cong J/J_{k}$, where $J=\pi_1(M_L)$.
\label{lemma:iso}
\end{corollary}

An outline of the proof is as follows. Let $L_i$ be the $i^{\text{th}}$ component of $L$. The group $G/G_{k+1}$ has presentation given by
$$G/G_{k+1} \cong \langle x_1, \dots , x_m | [x_i, \lambda_i], F_{k+1} \rangle$$
where $\lambda_i$ is the longitude of $L_i$ and $x_i$ is a meridian of $L_i$.  The inclusion of $S^3-L$ into $M_L$ induces an epimorphism on fundamental groups that has kernel normally generated by $\lambda_1, \dots, \lambda_m$.  The fundamental group $J=\pi_1(M_L)$ is obtained from $G$ by setting the longitudes $\lambda_i$ to zero.  This gives the presentation  $J/J_k \cong \langle x_1, \dots , x_m | \lambda_i, F_{k+1} \rangle$.

Suppose that the map induced from inclusion from $G/G_{k+1}$ to $F/F_{k+1}$ is an isomorphism. Then $[x_i, \lambda_i] \in F_{k+1}$, and thus $\lambda_i \in F_k$ since $x_i$ is a generator of $F$.  It is apparent that $J/J_k \cong F/F_k$.

Conversely, if $J/J_k \cong F/F_k$ then the relations show that $\lambda_i \in F_k$ and thus $[x_i, \lambda_i] \in F_{k+1}$.  This gives that $G/G_{k+1} \cong F/F_{k+1}$. It follows that $\bar \mu_L(I)=0$ for $|I| \leq k$ if and only if $F/F_k \cong J/J_k$.

\section{Main Theorem}

Before now, little has been known about the relationship between Milnor's invariants and ($n$)-solvability.  The following theorem demonstrates a relationship between Milnor's invariants and links, $L$, with certain restrictions on the 4-manifold bounded by $M_L$.  This theorem will be used to find an obstruction that detects an infinite subgroup of $\mathcal{F}^m_{n.5}/\mathcal{F}^m_{n+1}$ as well as to show that $\mathcal{F}^m_{-0.5}/\mathcal{F}^m_{1}$ is a nonabelian group.

\begin{theorem}Suppose $L$ is a link whose zero-framed surgery, $M_L$ bounds an orientable and compact 4-manifold $W$ that satisfies
\begin{itemize}
 \item[1.] $H_1(M_L) \rightarrow H_1(W; \mathbb{Z})$ is an isomorphism induced by the inclusion map;
  \item[2.]$H_2(W;\mathbb{Z})$ has a basis consisting of connected compact oriented surfaces $\{L_i\}$ with $\pi_1(L_i) \subset \pi_1(W)^{(n)}$.
   \end{itemize}
   Then $\bar \mu_L (I) = 0$ where $|I| \leq 2^{n+2}-1$.
\label{theorem:newmain}
\end{theorem}

\begin{proof}
As mentioned above in Theorem~\ref{theorem:milnor}, $\bar \mu_L(I)$=0 for all $|I| \leq k$ for any link $L$ in $S^3$ if and only if $F/F_{k+1} \cong G/G_{k+1}$, where $F=F\langle x_1, \cdots, x_m\rangle$ and $G=\pi_1(S^3-L)$. Using Corollary~\ref{lemma:iso}, this is equivalent to $F/F_{k}$ being isomorphic to $J/J_{k}$ where $J=\pi_1(M_L)$.

  Suppose $L$ is an $m$ component link whose zero-framed surgery, $M_L$ bounds a 4-manifold $W$ with $H_2(W;\mathbb{Z})$ having a basis consisting of connected compact oriented surfaces $\{L_i\}$ with $\pi_1(L_i) \subset \pi_1(W)^{(n)}$.   Consider the following sequence of maps on $\pi_1$ induced by inclusion (we are viewing $F$ as the fundamental group of a wedge of $m$ circles)
$$\begin{CD}
F @>\phi_1>>G@>\phi_2>> J @>\phi_3>> E=\pi_1(W).\\
\end{CD}$$
The map $\phi_2$ is the surjection induced by the inclusion of $S^3-L$ into $M_L$ and has kernel normally generated by the longitudes.
The quotients of all of these groups by the $k^{th}$ terms of their lower central series gives another sequence of maps
$$\begin{CD}
F/F_k @>\overline{\phi_1}>>G/G_k@>\overline{\phi_2}>> J/J_k @>\overline{\phi_3}>> E/E_k.\\
\end{CD}$$
Since $\phi_2$ is surjective, the map $\overline{\phi_2}:G/G_k \rightarrow J/J_k$ is a surjection for all values of $k$.


Dwyer's Theorem~\cite{Dw} is of particular importance and is stated here for convenience.

\begin{theorem}[Dwyer's Integral Theorem]
Let $\phi:A\rightarrow B$ be a homomorphism that induces an isomorphism on $H_1(-;\mathbb{Z})$. Then for any positive integer $k$, the following are equivalent:
\begin{itemize}
\item[i.] $\phi$ induces an isomorphism $A/A_{k+1} \cong B/B_{k+1}$;
\item[ii.] $\phi$ induces an epimorphism $H_2(A;\mathbb{Z})/\Phi_k(A) \rightarrow H_2(B;\mathbb{Z})/\Phi_k(B)$.
\end{itemize}
where $\Phi_k(A)=\ker (H_2(A)\rightarrow H_2(A/A_k))$ for $k \geq 1$.\
\label{theorem:dwyer}
\end{theorem}

Consider the map induced by $\phi_3 \circ \phi_2 \circ \phi_1$
 \begin{equation}H_2(F;\mathbb{Z})/\Phi_k(F) \rightarrow H_2(E;\mathbb{Z})/\Phi_k(E)\label{map:dwyer}\end{equation}
where $E=\pi_1(W)$.  Showing (2.2) is a surjection is equivalent to showing $\phi:=\overline{\phi_3} \circ \overline{\phi_2} \circ \overline{\phi_1}:F/F_{k+1} \rightarrow E/E_{k+1}$ is an isomorphism.

Since $F$ is the free group on $m$ generators, $H_2(F;\mathbb{Z})=0$. The map of (2.2) is a surjection precisely when $\Phi_k(E)= H_2(E,\mathbb{Z}).$  We need to determine for which $k$ we have $\Phi_k(E)= H_2(E,\mathbb{Z}).$

Consider the following diagram

$$\begin{CD}
H_2(W_k)  @>p_*>> H_2(W)\\
@VVV  @VVV\\
H_2(E_k)  @>i_*>>H_2(E) @>\pi_*>>H_2(E/E_{2k-1})
\end{CD}$$\\
where $W_k$ is the covering space of $W$ that corresponds with the $k^{th}$ term of the lower central series of $\pi_1(W)$.
The vertical maps are surjections obtained from the exact sequence induced by the Hurewicz map (2.3).  The maps $p_*$, $ i_*$ and $\pi_*$ are the maps induced by the covering map $p$, inclusion and projection respectively.
\begin{equation}\pi_2(X) \rightarrow H_2(X) \rightarrow H_2(\pi_1(X)) \rightarrow 1.\label{seq:hurewicz}\end{equation}

By assumption, there is a basis of $H_2(W)$ consisting of surfaces, $\{L_i\}$. The group $H_2(E)$ is generated by the images of the $L_i$ since $H_2(W)\rightarrow H_2(E)$ is a surjection.

We claim that the map $i_*$ is a surjection.  This can be seen by viewing $H_2(E_k)$ as the second homology group for the covering space, $K(E_k,1)$ of the Eilenberg-Maclane space $K(E,1)$. Note that $K(E_k,1)$ is the covering space of $K(E,1)$ corresponding to the subgroup $E_k$ of $E$. When $k=2^n$ the images of $\{L_i, D_i\}$ in $H_2(E)$ will lift to $H_2(E_k)$ so $i_*:H_2(E_k) \rightarrow H_2(E)$ is surjective.

Cochran and Harvey~\cite[Lemma 5.4]{CH4} showed that the composition of the following maps
$$\begin{CD}
H_2(E_k)  @>i_*>>H_2(E) @>\pi_*>>H_2(E/E_{2k-1})
\end{CD}$$
is the zero map for all $k$.  Since $i_*$ is surjective, this implies that  $\pi_*$ is the zero map.  Hence $\Phi_{2k-1}(E)=H_2(E)$ and Dwyer's theorem gives an isomorphism $H_2(F)/\Phi_k(F) \rightarrow H_2(E)/\Phi_k(E)$ induced by $\phi_3 \circ \phi_2 \circ \phi_1$ for $k=2^n$.  In turn, this gives isomorphism $F/F_{2k}\cong E/E_{2k}$ when $k=2^n$.
 Thus $\hat \phi := \overline{\phi_2}\circ \overline{\phi_1}: F/F_{2^{n+1}} \rightarrow J/J_{2^{n+1}}$ is a monomorphism.  Since $\overline{\phi_1}$ is a map $F/F_k \rightarrow F/\langle \text{relations}, F_k \rangle$ and $\phi_2$ is a surjection, by Milnor's presentation (2.1), $\hat \phi$ is a surjection and thus an isomorphism.

By Theorem~\ref{theorem:milnor} and Corollary~\ref{lemma:iso}, the $\bar \mu$-invariants of length less than or equal to $2^{n+1}$ vanish for ($n$)-solvable links.

This result can be improved slightly.  Let $g=(\hat \phi)^{-1}$ be a specified isomorphism.  Let $f$ be the composite of the following maps
\[
\begin{diagram}
\node{J} \arrow{e,t}{\pi_J}
\node{J/J_{2^{n+1}}} \arrow{e,t}{g}
\node{F/F_{2^{n+1}}}
\end{diagram}
\]
where $\pi_J$ is the canonical quotient map.  Consider the following commutative diagram
\[
\begin{diagram}
\node{E/E_{2^{n+1}}}
\node{J/J_{2^{n+1}}} \arrow{w,tb}{\phi}{\cong}
\node{F/F_{2^{n+1}}} \arrow{w,tb}{g^{-1}}{\cong}\\
\node{E} \arrow{n,r}{\pi_E}
\node{J} \arrow{n,l}{\pi_J} \arrow{ne,b}{g \circ \pi_J=f} \arrow{w}
\end{diagram}
\]
where $\phi$ is the isomorphism between $J/J_{2^{n+1}}$ and $E/E_{2^{n+1}}$ established earlier in the proof and $\pi_E$ is the canonical quotient map.  Thus we have an extension of $f$ to $E$, namely $\bar f=g \circ \phi^{-1} \circ \pi_E:E \rightarrow F/F_{2^{n+1}}$.
This gives the following commutative diagram.
\[
\begin{diagram}
\node{\pi_1(M_L)} \arrow{s,l}{i_*} \arrow{e,t}{f} \node{F/F_{2^{n+1}}}\\
\node{\pi_1(W)} \arrow{ne,b}{\bar f}
\end{diagram}
\]
The commutative diagram below on homology is achieved by the induced maps obtained from the above maps.
\[
\begin{diagram}
\node{H_3(M_L)} \arrow{s,l}{i_*} \arrow{e,t}{f} \node{H_3(F/F_{2^{n+1}})}\\
\node{H_3(W)} \arrow{ne,b}{\bar f}
\end{diagram}
\]
Since $M_L=\partial W$, $i_*:H_3(M_L)\rightarrow H_3(W)$ is the zero map.  Also, $\bar f :H_3(M_L) \rightarrow H_3(F/F_{2^{n+1}})$ is the zero map since the diagram commutes. Consider the following sequence of maps
$$H_3(F/F_{2k-1}) \overset{h_{k-1}}{\longrightarrow} H_3(F/F_{2k-2}) \overset{h_{k-2}}{\longrightarrow} \cdots \overset{h_2}{\longrightarrow} H_3(F/F_{k+1}) \overset{h_1}{\longrightarrow} H_3(F/F_k)$$
where $h_i:H_3(F/F_{k+i}) \rightarrow H_3(F/F_{k+i-1})$ and $k=2^{n+1}$.  The image of the fundamental class under the map $H_3(M_L)\rightarrow H_3(F/F_m)$ will be denoted by $\theta_{m}(M_L,f)$.

 We will use the following two results of Cochran, Gerges and Orr~\cite{CGO}.  These results rely heavily on deep work of Igusa and Orr~\cite{IO}.

\begin{lemma}[Cochran-Gerges-Orr]
$\theta_{m}(M_L, f) \in \text{Image}(\pi_*:H_3(F/F_{m+1}) \rightarrow H_3(F/F_m))$ if and only if there is some isomorphism $\tilde f:J/J_{m+1} \rightarrow F/F_{m+1}$ extending $f$ such that $\pi_*(\theta_{m+1}(M_L, \tilde f))=\theta_m(M_L,f)$.
\label{lemma:extension}
\end{lemma}

\begin{corollary}[Cochran-Gerges-Orr]
The map $H_3(F/F_{2m-1}) \rightarrow H_3(F/F_m)$ is the zero map.  Any element in the kernel of $H_3(F/F_{m+j}) \rightarrow H_3(F/F_m)$, $j\leq m-1$, lies in the image of $H_3(F/F_{2m-1}) \rightarrow H_3(F/F_{m+j}).$
\label{cor:kerim}
\end{corollary}

Since $\theta_{2^{n+1}}(M_L,f)=$[$0$] and [0] is always in the image of a homomorphism, there is an extension of $f$ to an isomorphism $\tilde f:J/J_{k+1} \rightarrow F/F_{k+1}$ with $h_1(\theta_{k+1}(M_L,\tilde f))=\theta_k(M_L,f)=0$ by Lemma~\ref{lemma:extension} and $\theta_{k+1}(M_L,\tilde f)$ is in the kernel of $h_1$.  Then $\theta_{k+1}(M_L,\tilde f)$ lies in the image of $H_3(F/F_{2k-1})\rightarrow H_3(F/F_{k+1})$ by Corollary~\ref{cor:kerim}.  In other words, it lies in the image of the map $h_2 \circ h_3 \circ \cdots \circ h_{k-1}$ and in turn lies in the image of $h_2$.  By Lemma~\ref{lemma:extension}, there is an extension of $\tilde f$ that is an isomorphism between $J/J_{k+2}$ and $F/F_{k+2}$.  By continuing this process, an isomorphism between $J/J_{2k-1}$ and $F/F_{2k-1}$ with $k=2^{n+1}$ is obtained. Thus we have that the $\bar \mu$-invariants of lengths less than or equal to $2^{n+2}-1$ of our link vanish.  This concludes the proof of Theorem~\ref{theorem:newmain}.
\end{proof}

We end this section with several remarks involving the limitations and generalizations of Theorem~\ref{theorem:newmain}.  We first notice that the proof of Theorem~\ref{theorem:newmain} did not rely on the intersection form seen in the definition of ($n$)-solvability.  As a consequence, we obtain the following corollary.

\begin{corollary}
If $L$ is an $(n)$-solvable link (or string link), then $\bar \mu_L(I)=0$ for $|I| \leq 2^{n+2}-1$.
\label{theorem:main}
\end{corollary}

The converse of Corollary~\ref{theorem:main} is false.  Consider the Whitehead link $= W$ in Figure~\ref{fig:whitehead}.  The first nonvanishing $\bar \mu$-invariant occurs at length four.  One of these invariants is $\bar \mu_W(1122)=\pm1$, depending on orientation.  The figure eight knot, $4_1$, may be obtained as the result of band summing the two components of $W$.  It is known that this knot is not $(0)$-solvable since its Arf invariant is nonzero~\cite{COT}. It is known that if a link is ($n$)-solvable, then the result of bandsumming any two components is ($n$)-solvable (see ~\cite{OT} for a complete proof).  This implies that the Whitehead link is not $(0)$-solvable.

\begin{figure}[h!]
  \centering
  \includegraphics[height=2.5cm]{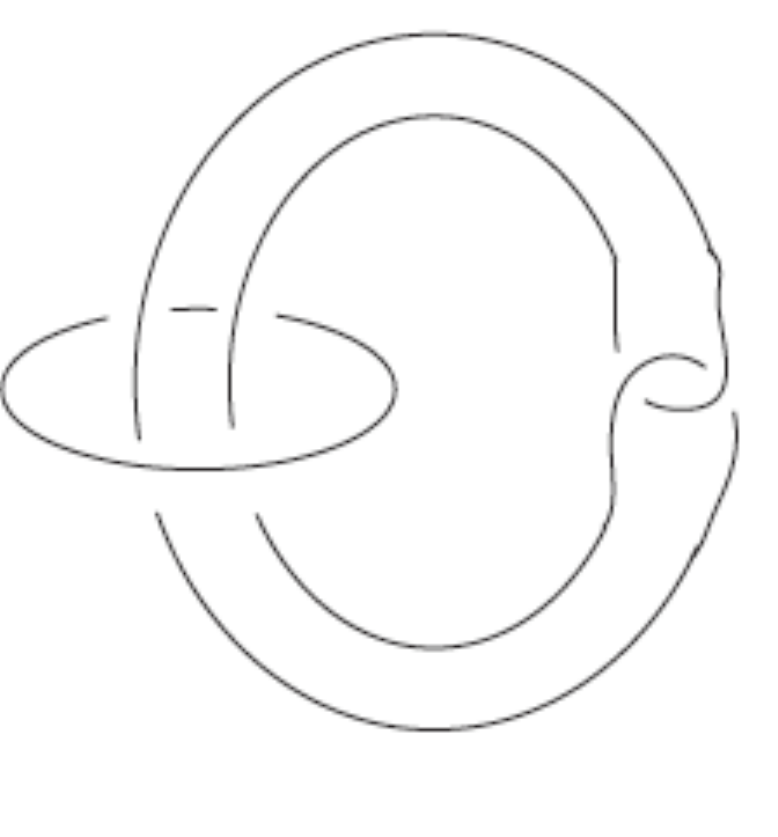}\qquad \includegraphics[height=2.5cm]{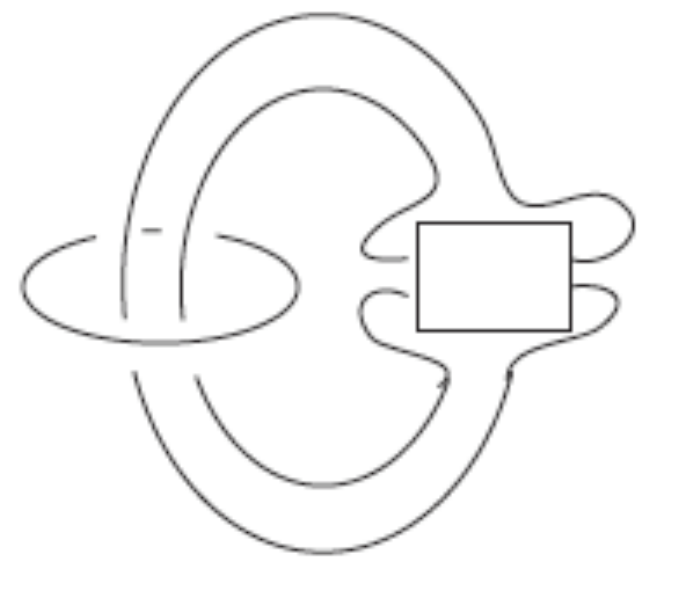}\put(-25,35){$n$}
  \caption[The Whitehead link and the $n$-twisted Whitehead link]%
  {The Whitehead Link and the $n$-twisted Whitehead link}
  \label{fig:whitehead}
\end{figure}

We also remark that the result of Corollary~\ref{theorem:main} is sharp in the sense the length of vanishing $\bar \mu$-invariants cannot be extended.  Consider the $n$-twisted Whitehead link in Figure~\ref{fig:whitehead}.  The number $n$ represents the number of full twists.  When $n$ is even, this link is band pass equivalent to the trivial link.  We will see in section 3 that this implies the link is $(0)$-solvable.  However, $\bar \mu(1122)=-n$ and $\bar \mu(1212)=2n$ are the first nonvanishing $\bar \mu$-invariants. There are examples of ($n$)-solvable links $L$ with $\bar \mu_L(I)=\pm 1$ for some $|I|=2^{n+1}$ (see section 5 for examples using Bing doubling).

Recently, Cochran, Harvey and Horn definition several new filtrations of $\mathcal{C}^m$~\cite{CHH}, namely the $n$-positive, $n$-negative, and $n$-bipolar filtrations.  The proof of Theorem~\ref{theorem:newmain} also can be applied to all these filtations as well.

\begin{corollary}
If $L$ is an $n$-positive, $n$-negative, or $n$-bipolar link, then $\bar \mu_L(I)=0$ for $|I| \leq 2^{n+2}-1$.
\label{cor:newresult}
\end{corollary}
\section{Bing Doubling and Solvability}

The goal of this section is to understand the effect that Bing doubling has on solvability. Bing doubling is a doubling operator performed on knots and links. If $K$ is a knot, then the two-component link in Figure~\ref{fig:BD(K)} is the Bing double of $K$, denoted $BD(K)$.  If $L$ is an $m$-component link, $BD(L)$ denotes be the $2m$-component link obtained by Bing doubling every component of $L$.

\begin{figure}[h!]
  \centering
  \includegraphics[width=2.5cm]{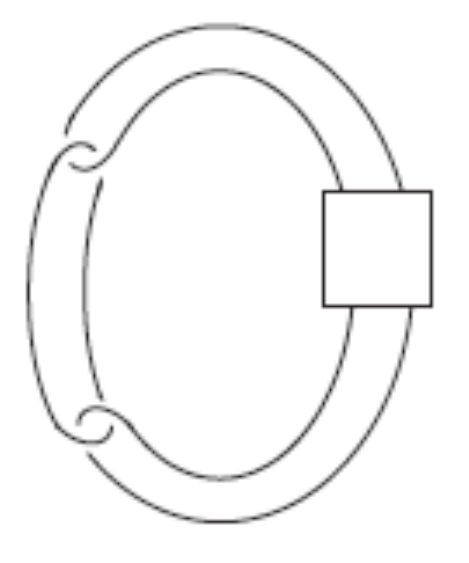}\put(-18.5,43){$K$}
  \caption[Bing doubling a knot]%
  {The Bing double of a knot $K$, BD(K)}
  \label{fig:BD(K)}
\end{figure}

Bing doubling can also be viewed as multi-infection by a string link. Let $L=L_1 \cup L_2 \cup \cdots \cup L_m$ in $S^3$ be an $m$ component link in $S^3$.  Let $L_{BD}$ be the $2m$-component link pictured in Figure~\ref{fig:trivialLBD} that is isotopic to the $2m$-component trivial link.
Then there is a handlebody, $H$, in $S^3-L_{BD}$ which is the exterior of a trivial string link with $m$ components (see Figure~\ref{fig:trivialLBD} for an example).  The $\eta_i$ are curves in $S^3-L_{BD}$ and are the canonical meridians of the trivial string link.

\begin{figure}[h!]
  \centering
  \subfigure[The trivial link $L_{BD}$]{\includegraphics[height=2.5cm]{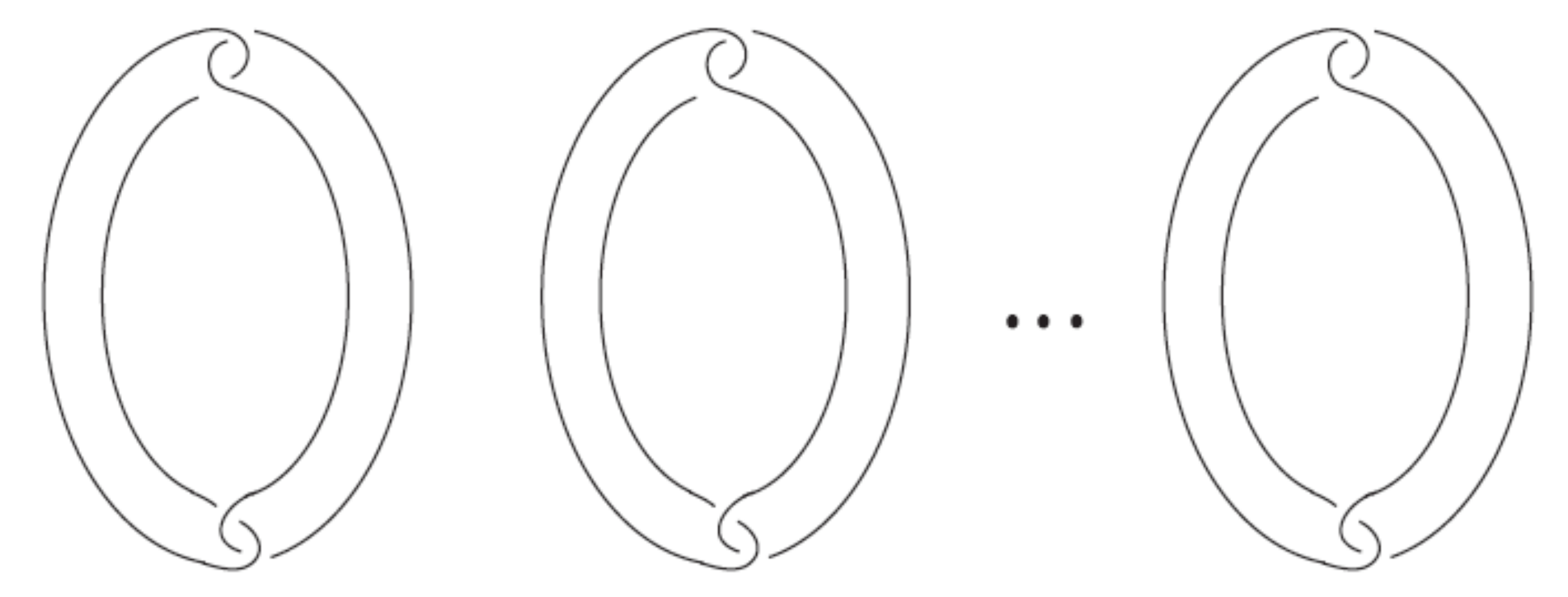}}\qquad \qquad
  \subfigure[A handlebody in $S^3-L_{BD}$]{\includegraphics[height=2.5cm]{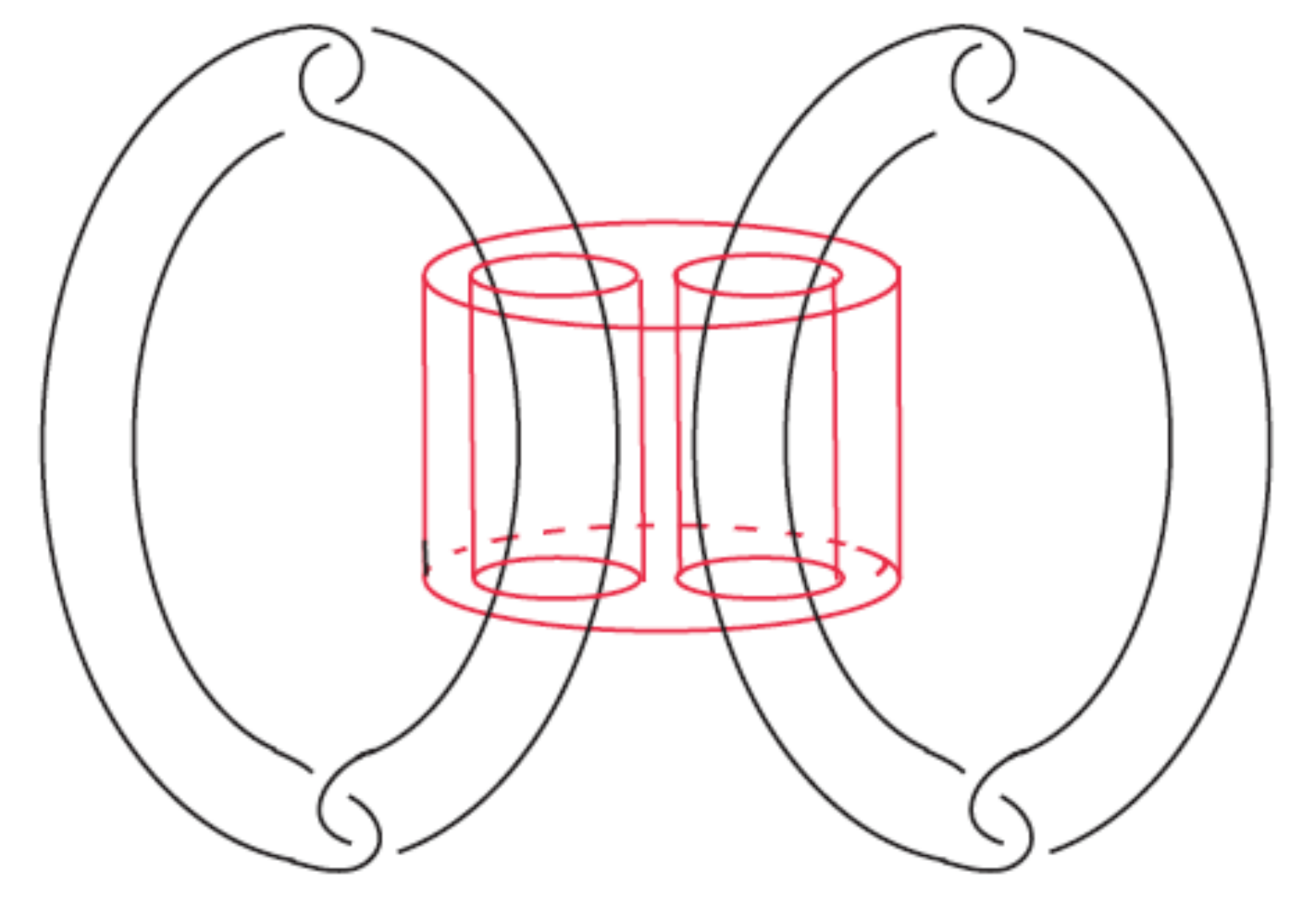}}
  \caption[The trivial link $L_{BD}$ and a handlebody in $S^3-L_{BD}$]
  { }
  \label{fig:trivialLBD}
\end{figure}

Take a string link $J$, such that $\hat J$ is isotopic to $L$.  There are an infinite number of string links that meet this criterion.  Then
$$BD(L)=((S^3-L_{BD})-H) \cup_\phi (D^2 \times I - J)$$
where $\phi$ maps $l_i \mapsto \gamma_i$ and $\mu_i \mapsto \eta_i^{-1}$ and the $l_i$, $\gamma_i$, $\mu_i$ and $\eta_i$ are depicted in Figure~\ref{fig:identifysl}.

\begin{figure}[h!]
\centering
\subfigure[Exterior of the trivial string link, the handlebody $H$]{\includegraphics[height=4cm]{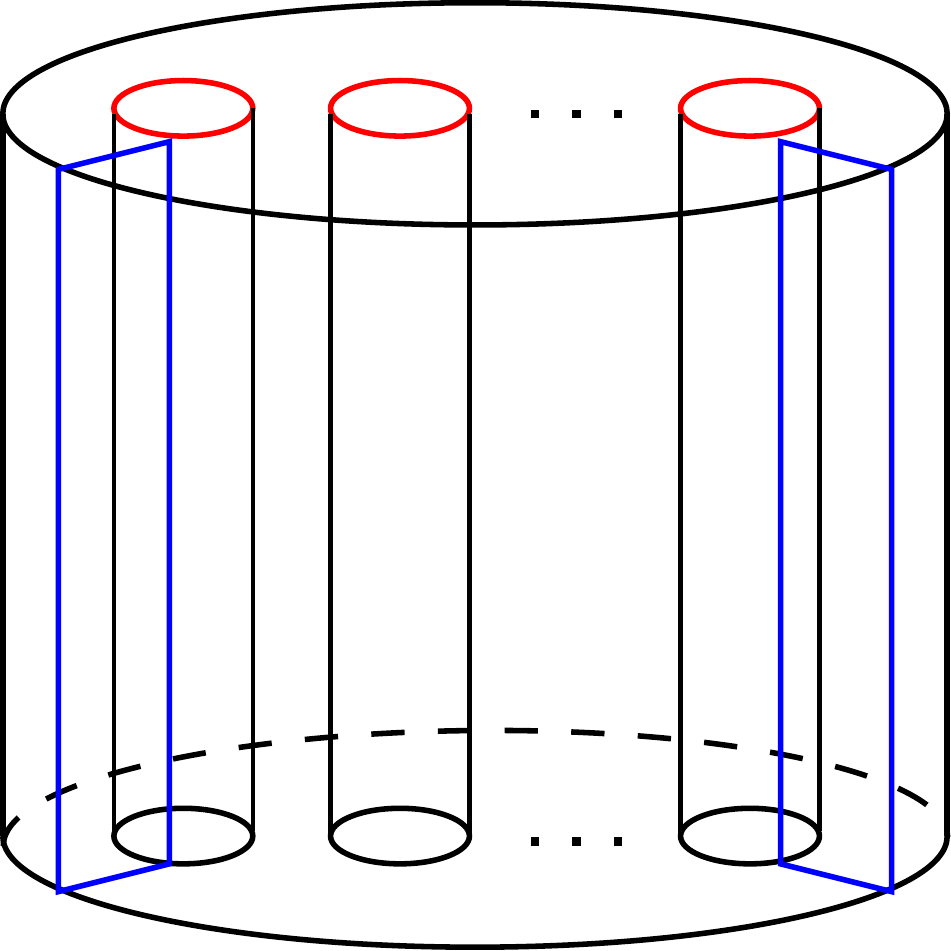}
\put(-98,116){\textcolor{red}{$\eta_1$}}
\put(-73,117){\textcolor{red}{$\eta_2$}}
\put(-30,116){\textcolor{red}{$\eta_n$}}
\put(-125,60){\textcolor{blue}{$l_1$}}
\put(5,60){\textcolor{blue}{$l_n$}}}\qquad \qquad \qquad
\subfigure[Exterior of the string link $J$]{\includegraphics[height=4cm]{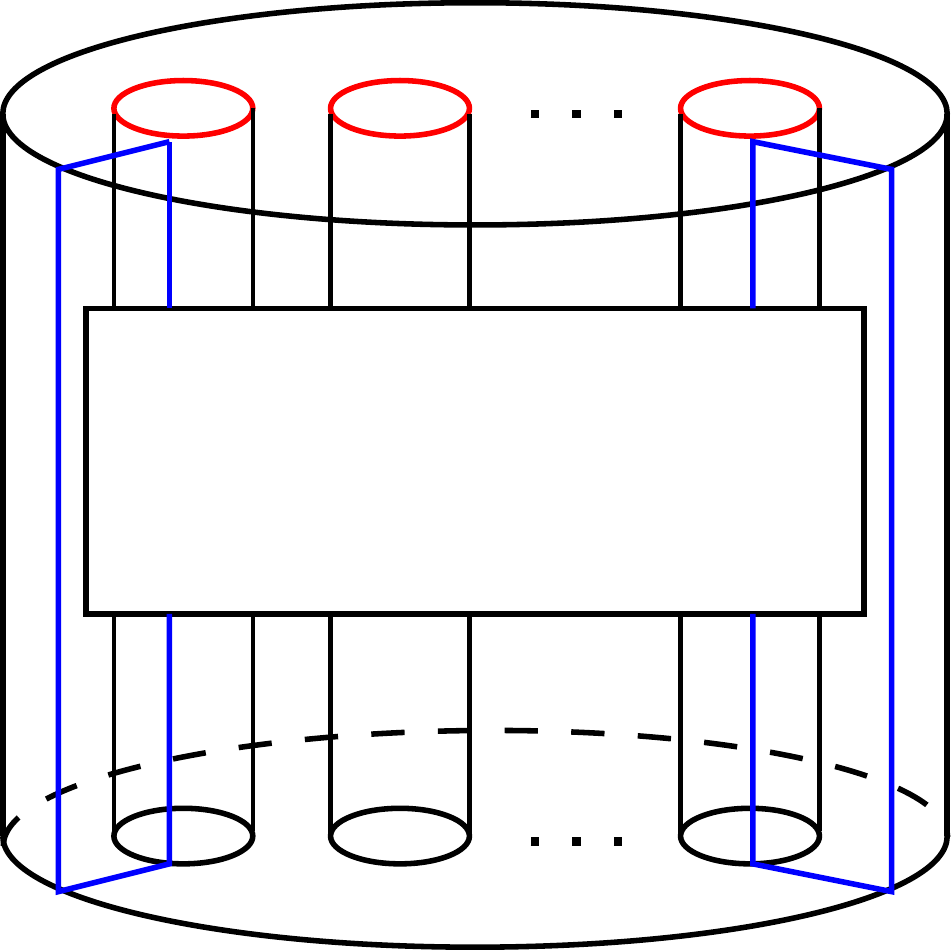}
\put(-98,116){\textcolor{red}{$\mu_1$}}
\put(-73,117){\textcolor{red}{$\mu_2$}}
\put(-30,116){\textcolor{red}{$\mu_n$}}
\put(-125,60){\textcolor{blue}{$\gamma_1$}}
\put(5,60){\textcolor{blue}{$\gamma_n$}}
\put(-60,55){$J$}}
\caption[Longitudes and meridians of string links]
{The longitudes $l_i, \gamma_i$ and meridians $\mu_i, \eta_i$}
\label{fig:identifysl}
\end{figure}
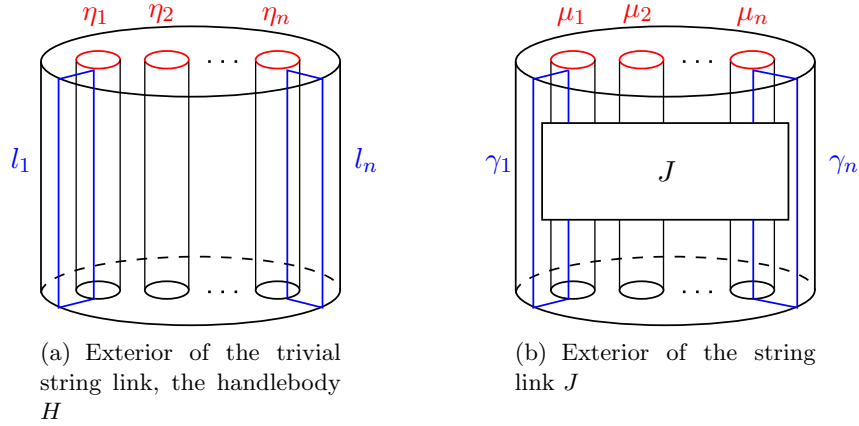

We will consider geometric moves that can be performed on knots and links and determine their effects on solvability.
The first move we consider is the band pass move, illustrated in Figure~\ref{fig:bandpass}.
 \begin{figure}[h!]
  \centering
  {
  \includegraphics[height=3cm]{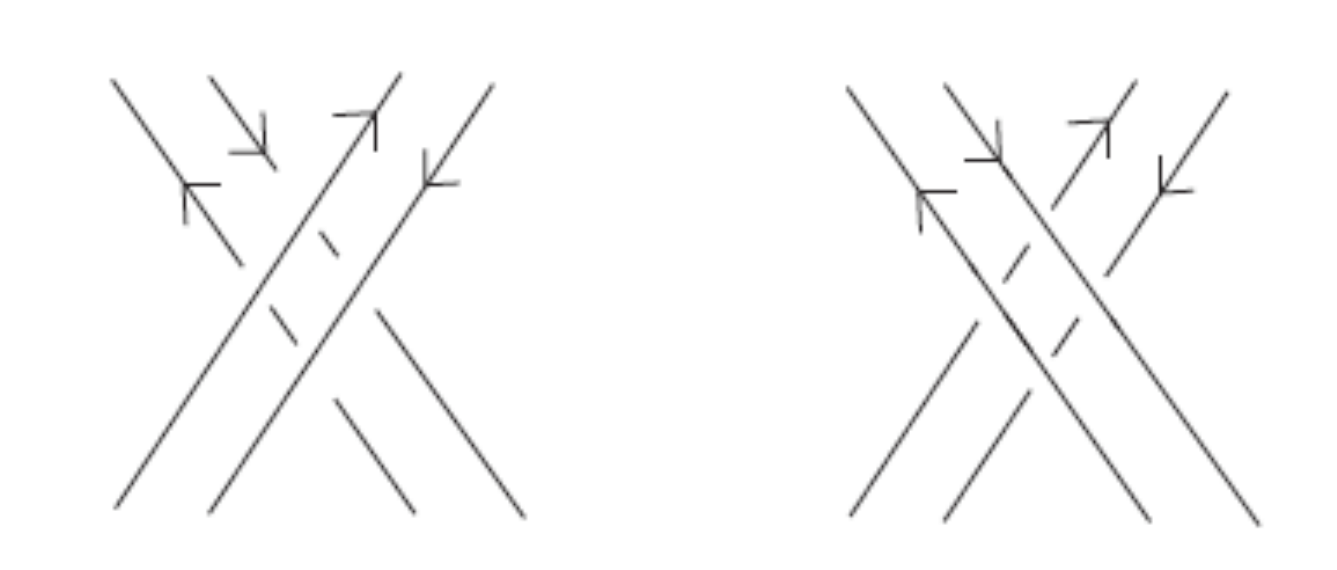}
  \put(-107,40){$\longleftrightarrow$}}
  \caption[Band pass move]%
  {Band pass move}
  \label{fig:bandpass}
\end{figure}
\begin{remark}
Lemmas~\ref{lemma:T11} and~\ref{lemma:T12} are results of Taylor Martin and complete proofs can be found in~\cite{TM}.
\end{remark}

\begin{lemma}[Martin]
A band pass move preserves $(0)$-solvability.
\label{lemma:T11}
\end{lemma}

\begin{proposition}If $L$ is any link of $m$ components, then $BD(L)$ is $(0)$-solvable.\end{proposition}

\begin{proof} Let $L$ be an $m$-component link in $S^3$.  The Bing double, $BD(L)$ is band pass equivalent to the trivial link of $2m$ components, arising from the fact that any link can be transformed into the trivial link by a finite number of crossing changes.  Since the trivial link is (0)-solvable and band pass moves preserve (0)-solvability, $BD(L)$ is (0)-solvable.
\end{proof}

We will also consider several other geometric moves, illustrated in Figures~\ref{fig:deltaclasp} and ~\ref{fig:doublemoves}.

\begin{figure}[h!]
  \centering
  \includegraphics[height=2cm]{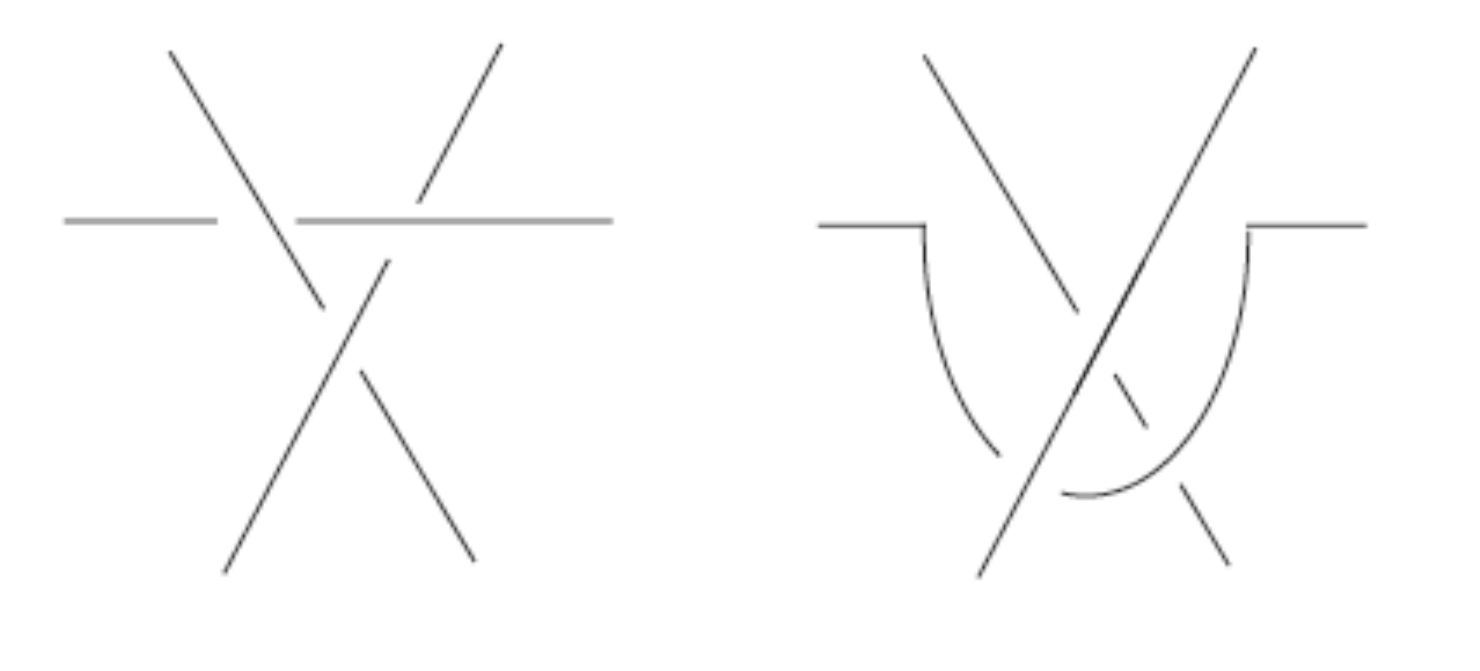}\qquad \includegraphics[height=2cm]{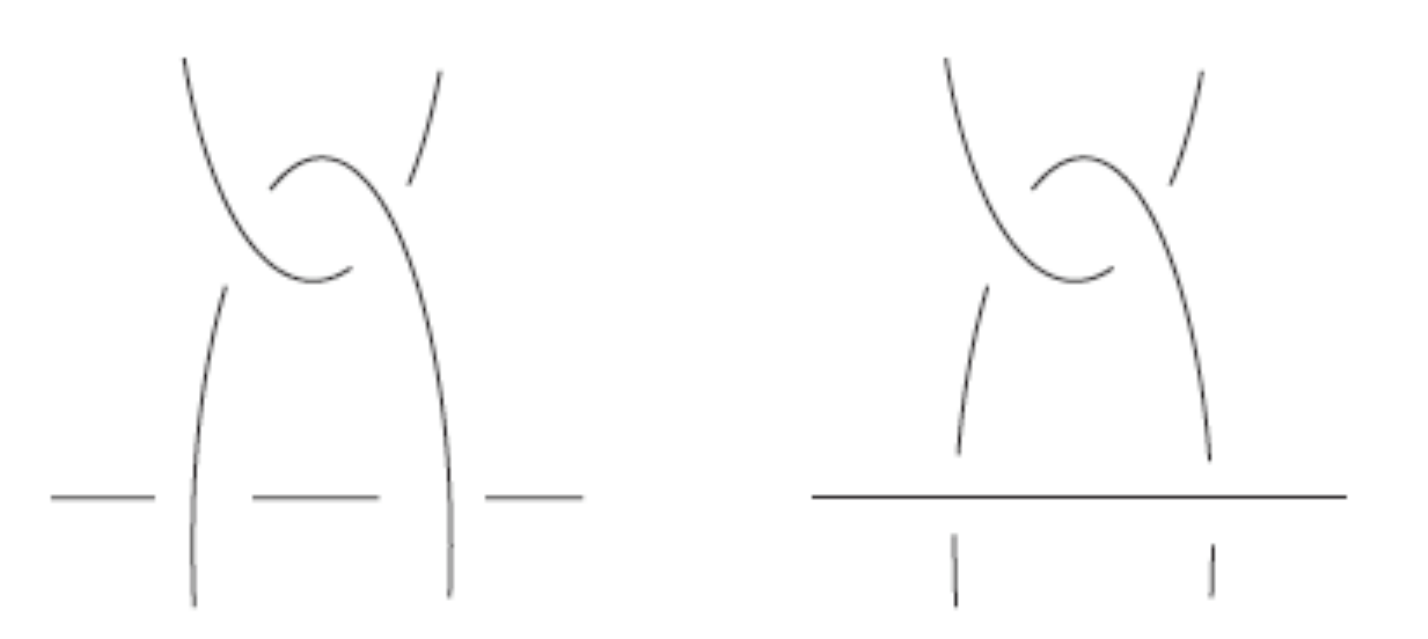}\put(-225,30){$\longleftrightarrow$}
  \put(-75,30){$\longleftrightarrow$}
  \caption[Delta and half-clasp moves]%
  {The delta move and the half-clasp move}
  \label{fig:deltaclasp}
\end{figure}

\begin{figure}[h!]
  \centering
  \includegraphics[height=3cm]{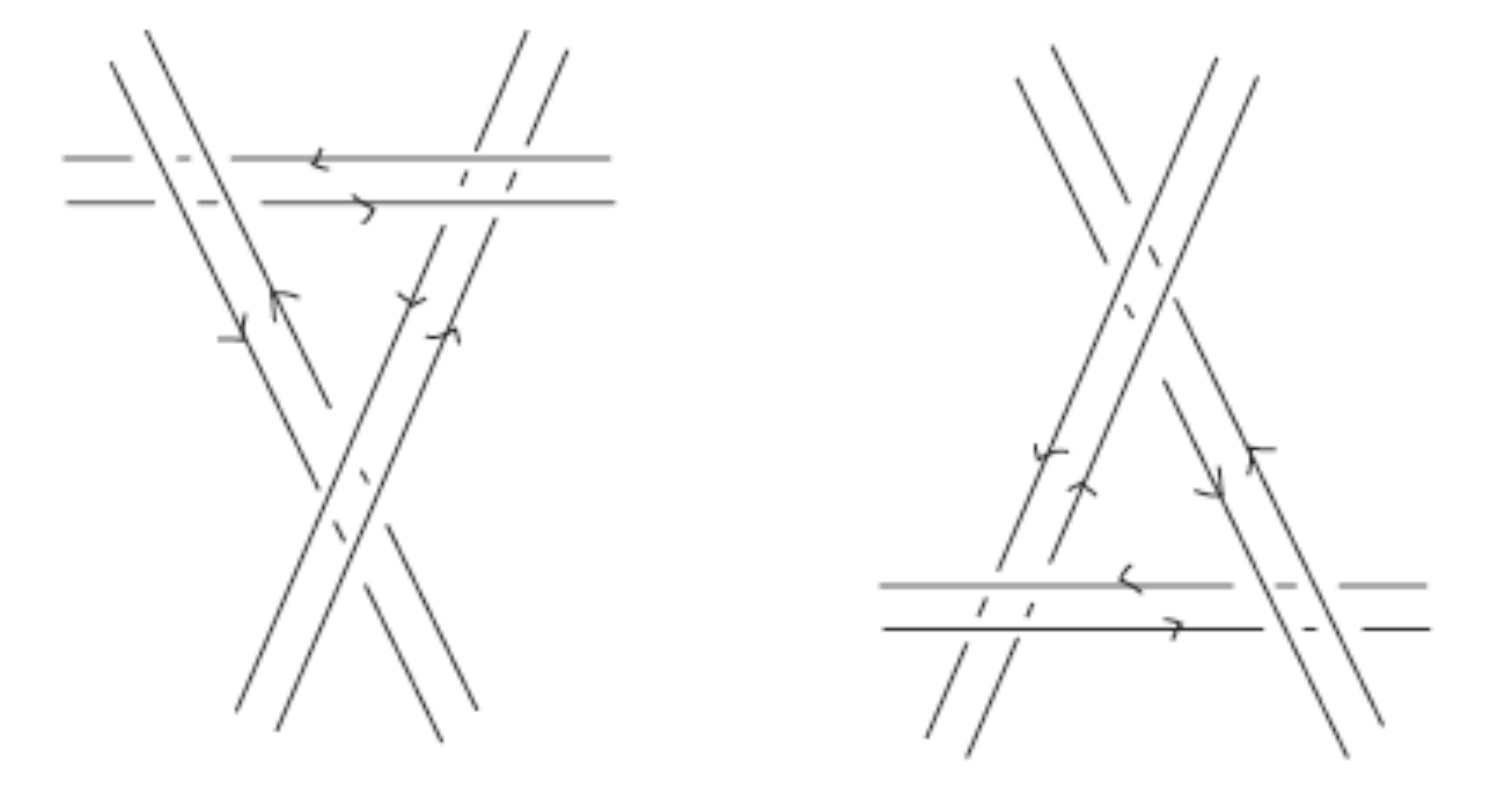}\qquad \qquad \includegraphics[height=3cm]{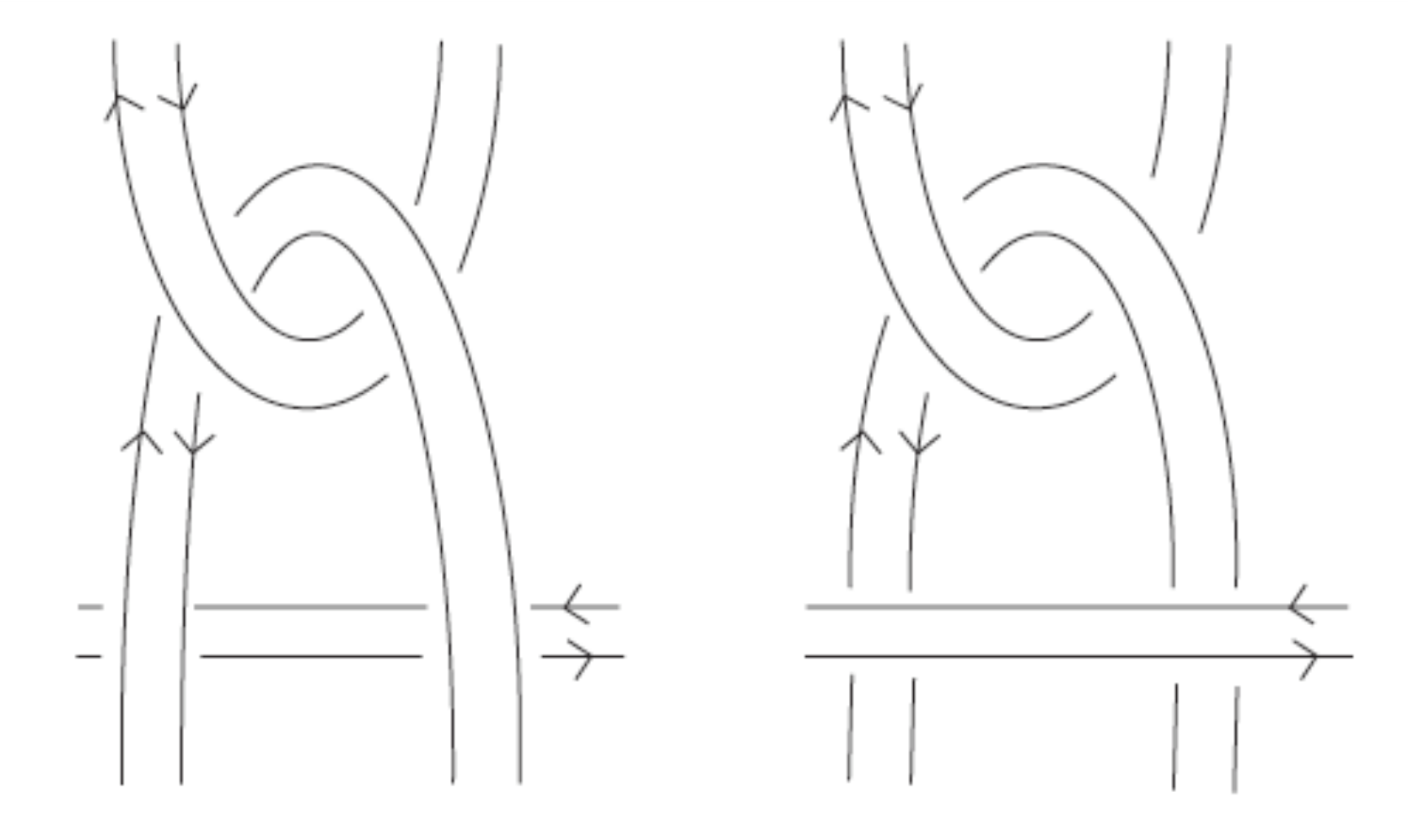}
  \put(-82,50){$\longleftrightarrow$}
  \put(-275,50){$\longleftrightarrow$}
  \caption[Double delta and double half-clasp moves]%
  {The double delta move and the double half-clasp move}
  \label{fig:doublemoves}
\end{figure}

\begin{lemma}
The delta move can be realized as a half-clasp move.  Moreover, the double delta move can be realized by a double half-clasp move.
\end{lemma}
\begin{proof}
The images in Figure~\ref{fig:isoseq} illustrate how to use isotopy and a half-clasp move to achieve the delta move.  This result is easily adaptable for the double of the moves.

\begin{figure}[h!]
  \centering
  \includegraphics[height=3cm]{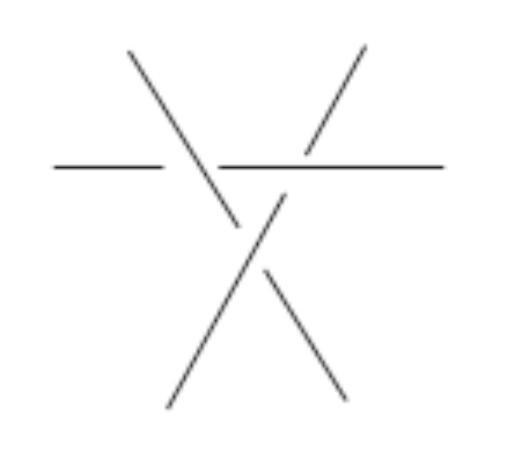}\qquad \includegraphics[height=3cm]{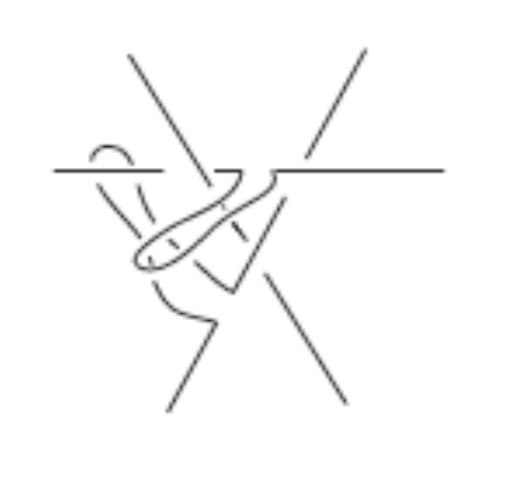} \qquad \includegraphics[height=3cm]{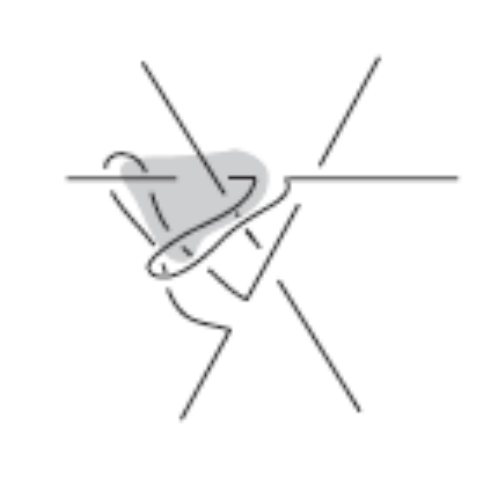}
   \put(-235,60){\small isotopy}
  \put(-230,50){$\longrightarrow$}
  \put(-108, 50){$\longrightarrow$}\\
  \includegraphics[height=3cm]{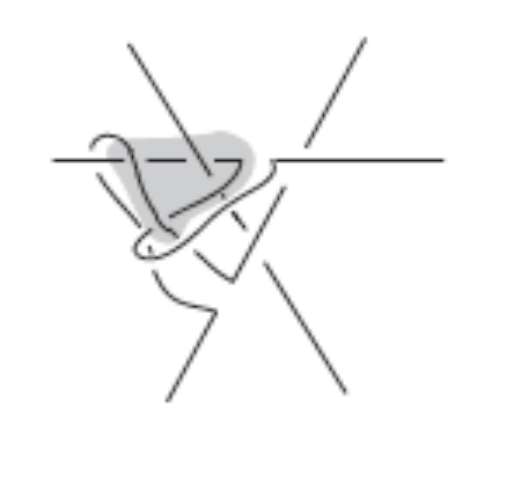}\qquad \includegraphics[height=3cm]{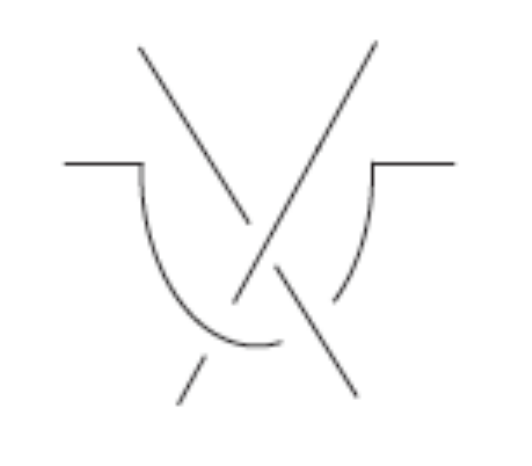}
  \put(-245,60){\small half}
  \put(-247,40){\small clasp}
  \put(-245,50){$\longrightarrow$}
  \put(-123,60){\small isotopy}
  \put(-118,50){$\longrightarrow$}
  \caption[Obtaining the delta move by a half clasp move]%
  {Sequence of isotopy and half-clasp moves to achieve the delta move.}
  \label{fig:isoseq}
\end{figure}
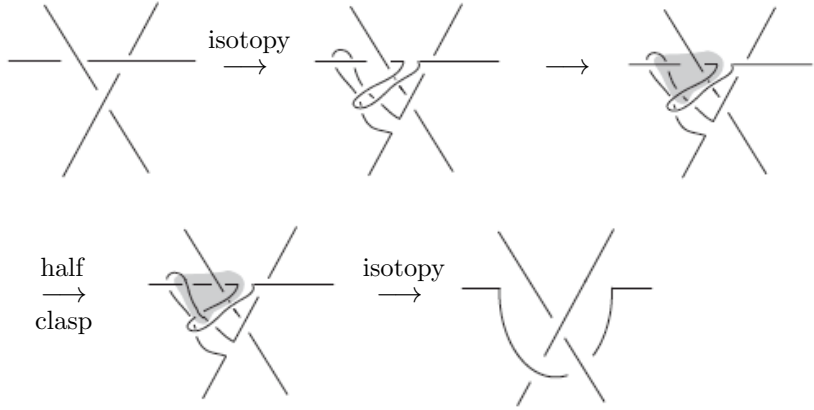

\end{proof}

Martin established a relationship between ($0.5$)-solvability and the double half-clasp move which is given in the following lemma.
\begin{lemma}[Martin]
The double half-clasp move preserves $(0.5)$-solvability.
\label{lemma:T12}
\end{lemma}

\begin{proposition} If $L\in \mathcal{F}^m_{-0.5}$, then $BD(L)$ is $(0.5)$-solvable.
\end{proposition}

\begin{proof} Suppose $L$ has all pairwise linking numbers equal to zero.  It was shown (~\cite{MN}, and~\cite{MA}) that two links are equivalent by delta moves if and only if they have the same pairwise linking numbers.  This result was generalized for string links~\cite{NS}.  Recall that in our construction of Bing doubling of a link, we chose a string link $J$ such that $\hat J$ is isotopic to $L$.  Since $\hat J$ has all pairwise linking numbers equal to zero by assumption, $J$ can be chosen to have all pairwise linking numbers equal to zero as a string link.

In the construction of Bing doubling we can see that the handlebody $H$ was replaced with the exterior of $J$.  As a result of this replacement, we have a new string link $\tilde J$.  Using double delta moves, we are able to get the trivial link (delta moves on $J$ will be double delta moves on $\tilde J$).  Since the double half-clasp move preserves (0.5)-solvability by Lemma~\ref{lemma:T12}, the double delta move will also preserve (0.5)-solvability.  Thus $BD(L)$ is (0.5)-solvable.
\end{proof}

\begin{proposition}
If $L$ is an $(n)$-solvable link, then $BD(L)$ is $(n+1)$-solvable. Moreover,
if $L$ is an $(n.5)$-solvable link, then $BD(L)$ is $((n+1).5)$-solvable.
\label{proposition:bingdouble}\end{proposition}

\begin{proof}
Suppose $L$ is an ($n$)-solvable link with $m$ components.  We will construct an ($n+1$)-solution for $BD(L)$ by first constructing a cobordism between $M_L$ and $M_{BD(L)}$.  Suppose $J$ is a string link such that $\hat J$ is isotopic to $L$.  Then $M_L=(D^2\times I -J) \cup (D^2 \times I-\text{trivial string link})$.  Consider $M_L \times [0,1]$ and $M_{L_{BD}} \times [0,1]$.  Recall that $L_{BD}$ was isotopic to the $2m$-component trivial link.  Let $V$ be the handlebody $D^2 \times I$. Glue $M_L \times \{1\}$ to $M_{L_{BD}} \times \{1\}$ by identifying $V \subset M_L \times \{1\}$ with $V\subset M_{L_{BD}} \times \{1\}.$  Call the resulting space $X$ (see Figure~\ref{fig:cobordism}).  Then $\partial X= M_L \coprod M_{L_{BD}} \coprod -M_{BD(L)}$.

\begin{figure}[h!]
  \centering
  \includegraphics[height=5cm]{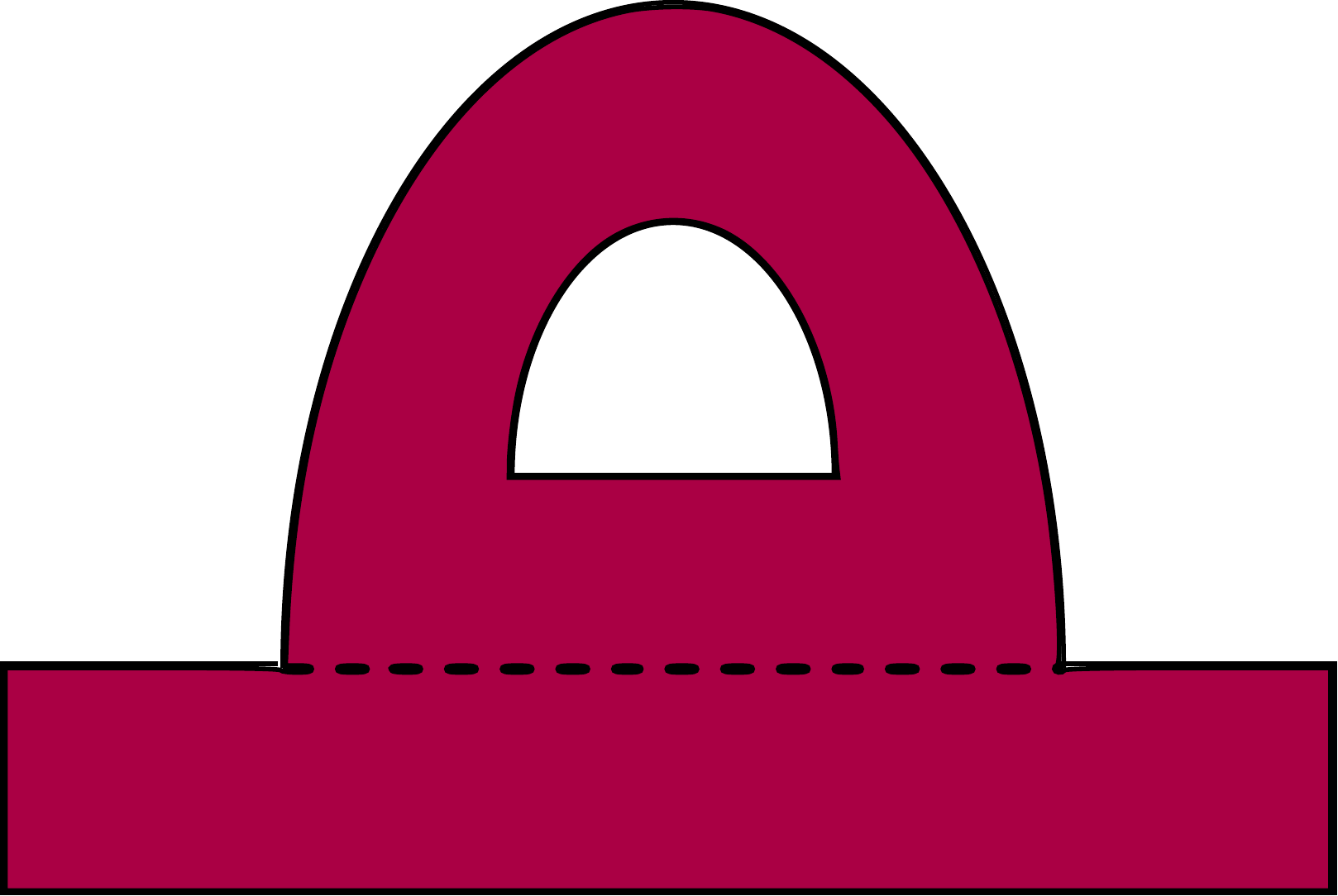}
  \put(-9,-13){$M_{L_{BD}} \times \{0\}$}
  \put(-9,40){$M_{L_{BD}} \times \{1\}$}
  \put(-111,42){$V$}
  \put(-45,110){$M_{L} \times \{1\}$}
  \put(-45,70){$M_{L} \times \{0\}$}
  \put (-75,70){$\longleftarrow$}
  \caption[Cobordism between $M_L$ and $M_{BD(L)}$]%
  {The space $X$}
  \label{fig:cobordism}
\end{figure}
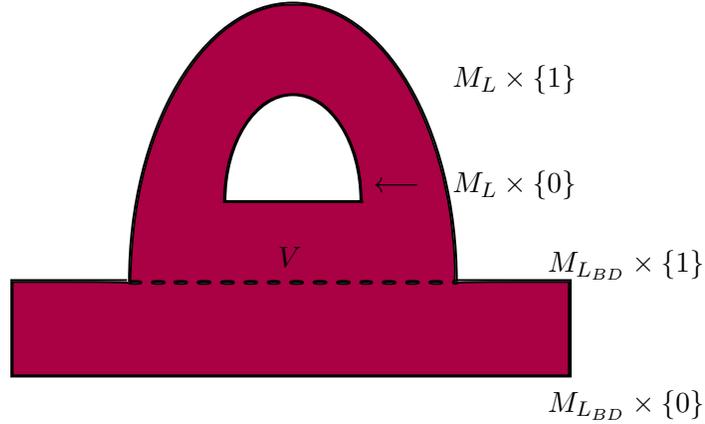

To proceed, we need the following lemma.
\begin{lemma}
With $X$ as above, the inclusion maps induce the following
\begin{itemize}
\item[i.] isomorphisms $H_1(M_{L_{BD}}) \rightarrow H_1(X)$ and $H_1(M_{BD(L)}) \rightarrow H_1(X)$;
\item[ii.] isomorphism $H_2(X) \cong H_2(M_{L_{BD}}) \oplus H_2(M_L)$.
\end{itemize}
\label{lemma:isobyinclusion}
\end{lemma}

\begin{proof}
Consider the following diagram of inclusion maps.

\[
\begin{diagram}
\node[2]{V} \arrow{sw,t}{i_1} \arrow{se,t}{i_2}\\
\node{M_{L_{BD}} \times [0,1]} \arrow{se,b}{j_2}
\node[2]{M_L \times [0,1]} \arrow{sw,b}{j_1} \\
\node[2]{X}
\end{diagram}
\]
Using Mayer Vietoris, the maps above induce the following long exact sequence (in reduced homology), where $I_*=(i_{1*}, i_{2*})$ and $J_*=j_{1*}-j_{2*}$ (the homology groups are with $\mathbb{Z}$ coefficients)
$$\begin{CD}
\cdots @>\partial_*>>H_2(V)  @>I_*>>H_2(M_{L_{BD}})\oplus H_2(M_L) @>J_*>>H_2(X)@>\partial_*>>\end{CD}$$
$$\begin{CD} H_1(V) @>I_*>>H_1(M_{L_{BD}})\oplus H_1(M_L)@>J_*>>H_1(X)@>\partial_*>>0.
\end{CD}$$

The homology group $H_1(V) \cong \mathbb{Z}^m$ is generated by the meridians, $\mu_i$ of the trivial string link.  Recall that the $\eta_i$s were defined in the construction of Bing doubling.  Now $i_{1*}(\mu_i)=0$ in $S^3-L_{BD} \subset M_{L_{BD}}$ since $\mu_i \sim \eta_i$ and $\eta_i$ is in a commutator subgroup.  Also, $i_{2*}(\mu_i)$ is of infinite order in $S^3-L \subset M_L$ since $\mu_i$ is identified with a meridian of $L$.  Hence $I_*$ is a monomorphism.  Thus the map $\partial_*:H_2(X) \rightarrow H_1(V)$ is the zero map.  By the properties of a long exact sequence,
$H_2(M_{L_{BD}}) \oplus H_2(M_L) \cong H_2(X).$

For the other part of the lemma, consider the first isomorphism theorem.  This gives
$H_1(X) \cong \dfrac{H_1(M_{L_{BD}}) \oplus H_1(M_L)}{\text{image}(I_*:H_1(\eta \times D^2) \rightarrow H_1(M_{L_{BD}})\oplus H_1(M_L))}.$
The image of $I_*$ is precisely $H_1(M_L)$.  Thus $H_1(X) \cong H_1(M_{L_{BD}}) \cong \mathbb{Z}^{2m}$. Now $H_1(M_{BD(L)})$ is generated by the meridians of $BD(L)$ which are isotopic (in $X$) to the meridians of $L_{BD}$.  This means that $H_1(X) \cong H_1(M_{BD(L)})$ which is the desired result.

\end{proof}
We now continue with the proof of the proposition.
Let $S=B^4-\mathbb{D}$ be a slice disk complement where $\mathbb{D} \subset B^4$ is a collection of disjoint and smoothly embedded disks with boundary $L_{BD}$.  Let $W$ be an ($n$)-solution for $L$ and let $E$ be the space obtained by attaching $W$ and $S$ to $X$ along $M_{L \times \{0\}}$ and $M_{L_{BD} \times \{0\}}$ respectively.  Thus $E$ is a 4-manifold with boundary $M_{BD(L)}$.

We claim that $E$ is an ($n+1$)-solution for $BD(L)$.  We begin by showing $E$ is an ($n$)-solution.  Let $\overline{E} = X \cup W$.  Consider the following long exact sequence (with $\mathbb{Z}$-coefficients in reduced homology) obtained by Mayer Vietoris

$$\begin{CD}
\cdots @>\partial_*>>H_2(M_L)  @>I_1>>H_2(X)\oplus H_2(W) @>I_2>>H_2(\overline{E})@>\partial_*>>\end{CD}$$
$$\begin{CD} H_1(M_L) @>I_1>>H_1(X)\oplus H_1(W)@>I_2>>H_1(\overline{E})@>\partial_*>>0.
\end{CD}$$

We have that inclusion induces an isomorphism $H_1(M_L) \cong H_1(W)$.  This together with the facts that $I_2$ on $H_1$ is surjective and $H_1(M_L) \rightarrow H_1(X)$ is the zero map, gives that $H_1(\overline{E}) \cong H_1(X)$.  From Lemma~\ref{lemma:isobyinclusion}, the inclusion maps induce an isomorphism $H_2(X)\cong H_2(M_{L_{BD}}) \oplus H_2(M_L)$.  Thus, by the first isomorphism theorem and definition of exact sequence, we obtain
$H_2(\overline{E})  \cong  \dfrac{H_2(X)\oplus H_2(W)}{\ker(I_2:H_2(X)\oplus H_2(W) \rightarrow H_2(\overline{E}))} \cong  H_2(M_{L_{BD}}) \oplus H_2(W)$.

Notice that $E=\overline{E} \cup S$.  Consider the following long exact sequence on homology given by Mayer Vietoris
$$\begin{CD}
\cdots @>\partial_*>>H_2(M_{L_{BD}})  @>\rho_1>>H_2(\overline{E})\oplus H_2(S) @>\rho_2>>H_2(E)@>\partial_*>>\end{CD}$$
$$\begin{CD} H_1(M_{L_{BD}}) @>\rho_1>>H_1(\overline{E})\oplus H_1(S)@>\rho_2>>H_1(E)@>\partial_*>>0.
\label{cd}\end{CD}$$

Using the facts, $H_1(M_{L_{BD}}) \cong H_1(X)$ induced by inclusion (Lemma~\ref{lemma:isobyinclusion}), $H_2(S)=0$, and $H_1(X) \cong H_1(\overline{E})$, we can again use the first isomorphism theorem to attain the following
$H_2(E)  \cong \dfrac{H_2(\overline{E})}{\ker(\rho_2)}\cong H_2(W).$
This shows that the second condition of ($n$)-solvability is satisfied for the 4-manifold $E$.

For the third condition, the inclusion map $i:W \hookrightarrow E$ gives $i_*(\pi_1(W)^{(n)}) \subseteq \pi_1(E)^{(n)}$.  Since no elements were added to the basis of $H_2(E)$, it has the same basis as $H_2(W)$.  Thus $\pi_1(L_i) \subset \pi_1(W)^{(n)} \subseteq \pi_1(E)^{(n)}$ and similarly for $\pi_1(D_i)$, where $\{L_i, D_i\}$ is a basis for $H_2(W)$.

To check the first condition of ($n$)-solvability, consider again the previous long exact sequence. The first isomorphism theorem tells us
\begin{center}
$\dfrac{H_1(M_{L_{BD}})}{\ker(\rho_1:H_1(M_{L_{BD}}) \rightarrow H_1(\overline{E})\oplus H_1(S))} \cong \text{image}(\rho_1).$
\end{center}
Since the $\ker(\rho_1) = \{0\}$, we have that image$(\rho_1)\cong H_1(M_{L_{BD}})$.  Now, $S$ is an ($n$)-solution for $M_{L_{BD}}$, and thus $H_1(M_{L_{BD}}) \cong H_1(S)$ induced by inclusion.  Using the first isomorphism a final time gives that $H_1(E) \cong H_1(\overline E)$.

By Lemma~\ref{lemma:isobyinclusion} and the above results, the first condition to being ($n$)-solvable is met and $E$ is an ($n$)-solution for $BD(L)$.

We claim further that $E$ is actually an ($n+1$)-solution.  Showing that $\pi_1(W) \subset \pi_1(E)^{(1)}$ (or more precisely, $i_*(\pi_1(W)) \subset \pi_1(E)^{(1)}$) is enough to imply that $\pi_1(W)^{(n)} \subset \pi_1(E)^{(n+1)}$.

Consider the following commutative diagram of maps where $i_*$ is induced by inclusion and both $p_{1_*}$ and $p_{2_*}$ are canonical quotient maps.

\[
\begin{diagram}
\node{\pi_1(W)} \arrow{e,t}{i_*} \arrow{s,l}{p_{1_*}} \arrow{se,t}{h} \node{\pi_1(E)} \arrow{s,r}{p_{2_*}}\\
 \node{H_1(W)=\frac{\pi_1(W)}{\pi_1(W)^{(1)}}} \arrow{e,b}{i_*}
\node{\frac{\pi_1(E)}{\pi_1(E)^{(1)}}=H_1(E)}
\end{diagram}
\]

Showing that $h\equiv  0$ is equivalent to showing that $\pi_1(W) \subset \pi_1(E)^{(1)}$.  Examining this further shows that $h \equiv 0$ if and only if $i_*:H_1(W)\rightarrow H_1(E)$ is the zero map, since our diagram commutes.  Consider the following commutative diagram.

\[
\begin{diagram}
\node{H_1(M_L)} \arrow{e,t}{\cong} \arrow{se,t}{p} \node{H_1(W)} \arrow{s,r}{i_*}\\
\node[2]{H_1(E)}
\end{diagram}
\]

To show that $i_* \equiv 0$ is equivalent to showing that the map $p:H_1(M_L) \rightarrow H_1(E)$ is the zero map.  Consider $[\mu_i] \in H_1(M_L)$ where the $\mu_i$s generate $H_1(M_L)$.  Under the map $p$, $[\mu_i] = [\eta_i] \in H_1(M_{L_{BD}}) \subset H_1(E)$ ($\eta_i \in S^3-L_{BD} \subset M_{L_{BD}}$).  But recall that $[\eta_i]$ lie in a commutator subgroup and thus $[\eta_i] = 0$ in homology, and thus $p$ is the zero map.  This shows that $E$ is an ($n+1$)-solution  and the desired result is achieved.

The case  when $L$ is ($n.5$)-solvable is similar.
\end{proof}

\section{Applications to $\{\mathcal{F}^m_n\}$}

In studying the ($n$)-solvable filtration, we often look at successive quotients of the filtration.  Recently, progress has been made towards understanding the structure of its quotients (see~\cite{Cha4},~\cite{CH2},~\cite{CHL3},~\cite{Ha2}).  Harvey first showed that $\mathcal{F}_n^m/\mathcal{F}^m_{n+1}$ is a nontrivial group that contains an infinitely generated subgroup~\cite{Ha2}.  She showed that this subgroup is generated by boundary links (links with components that bound disjoint Seifert surfaces).   Cochran and Harvey improved this result by showing that $\mathcal{F}^m_n/\mathcal{F}^m_{n.5}$ contains an infinitely generated subgroup~\cite{CH2}.  Again, this subgroup consists entirely of boundary links. Boundary links have vanishing $\bar \mu$-invariants at all lengths.

Using the relationship between Milnor's $\bar \mu$-invariants and ($n$)-solvability, given in Theorem~\ref{theorem:main}, we are able to establish new results that are disjoint from previous work.  Until now, nothing has been known about the ``other half" of the ($n$)-solvable filtration, namely $\mathcal{F}^m_{n.5}/\mathcal{F}^m_{n+1}$.

\begin{theorem}
$\mathcal{F}^m_{n.5}/\mathcal{F}^m_{n+1}$ contains an infinite cyclic subgroup for $m \geq 3*2^{n+1}$.
\label{theorem:nontrivial}
\end{theorem}

\begin{proof}
Let $BR=$ the Borromean Rings.  It is clear that $BR \in \mathcal{F}^3_{-0.5}$.  A direct calculation from the definition of Milnor's invariants will show that $\bar \mu_{BR}(123)=\pm1$ depending on orientation.  Using Theorem~\ref{theorem:main}, $BR$ is a nontrivial link in $\mathcal{F}^3_{-0.5}/\mathcal{F}^3_0$.  When we apply the Bing double operator on $BR$, denoted $BD(BR)$ (see Figure~\ref{fig:BBBR}), this new link is in $\mathcal{F}^6_{0.5}$ by Proposition~\ref{proposition:bingdouble}.  However, the first nonvanishing $\bar \mu$-invariant is $\bar \mu_{BD(BR)}(I)= \pm 1$ for a certain $I$ with $|I|=6$ (see Chapter 8 in~\cite{C4} for the specific details).  Hence $BD(BR)$ is not (1)-solvable by Theorem~\ref{theorem:main}.  Then $BD(BR)$ is nontrivial in $\mathcal{F}^6_{0.5}/\mathcal{F}^6_1$ since it has a nonvanishing $\bar \mu$-invariant.

We can perform the Bing doubling operation on this new link to form $BD(BD(BR))$, or more simply, $BD_2(BR)$ (see Figure~\ref{fig:BDBDBR}).  Using Proposition~\ref{proposition:bingdouble}, $BD_2(BR)$ is nontrivial in $\mathcal{F}^{12}_{1.5}$.  Looking at its $\bar \mu$-invariants, we will have that $\bar \mu_{BD_2(BR)}(I)=\pm 1$ for a certain $I$ of length 12 and our link cannot be (2)-solvable, again by Theorem~\ref{theorem:main}.  Therefore $BD_2(BR)$ is nontrivial in $\mathcal{F}^{12}_{1.5}/\mathcal{F}^{12}_2$.  We can continue this process to have $BD_{n+1}(BR)$ nontrivial in $\mathcal{F}^m_{n.5}/\mathcal{F}^m_{n+1}$ for $m\geq 3*2^{n+1}$.

We claim that $BD_{n+1}(BR)$ will have infinite order in $\mathcal{F}^m_{n.5}/\mathcal{F}^m_{n+1}$. Orr showed the first nonvanishing $\bar \mu$-invariant is additive ~\cite{O1}.  Consider an arbitrary string link $L$ with the following properties instead of the specific link $BD_{n+1}(BR)$ for the moment.  Suppose that $\bar \mu_{L}(I)=0$ and that $\bar \mu_{L}(J) \not =0$ for $|J|= |I|+1$.  Then
$\bar \mu_{\widehat {LL}}(J)=\bar \mu_{\widehat {L}}(J)+\bar \mu_{\widehat {L}}(J)=2\bar \mu_{\widehat {L}}(J).$
If we were to take the closure of the stack of $n$ copies of $L$, denoted $\widehat {nL}$, we would obtain
$\bar \mu_{\widehat {nL}}(J)=n\bar \mu_{\hat L}.$

 This gives that $L$ generates an infinite cyclic subgroup $\mathbb{Z}$.  In our case, since $BD_{n+1}(BR)$ has a nonzero $\bar \mu$-invariant of length $3*2^{n+1}$, the same reasoning can be used to show that it generates an infinite cyclic subgroup.
\end{proof}
\begin{figure}[h!]
\centering
\subfigure[$BD(BR)$]{\label{fig:BBBR}\includegraphics[height=4.7cm]{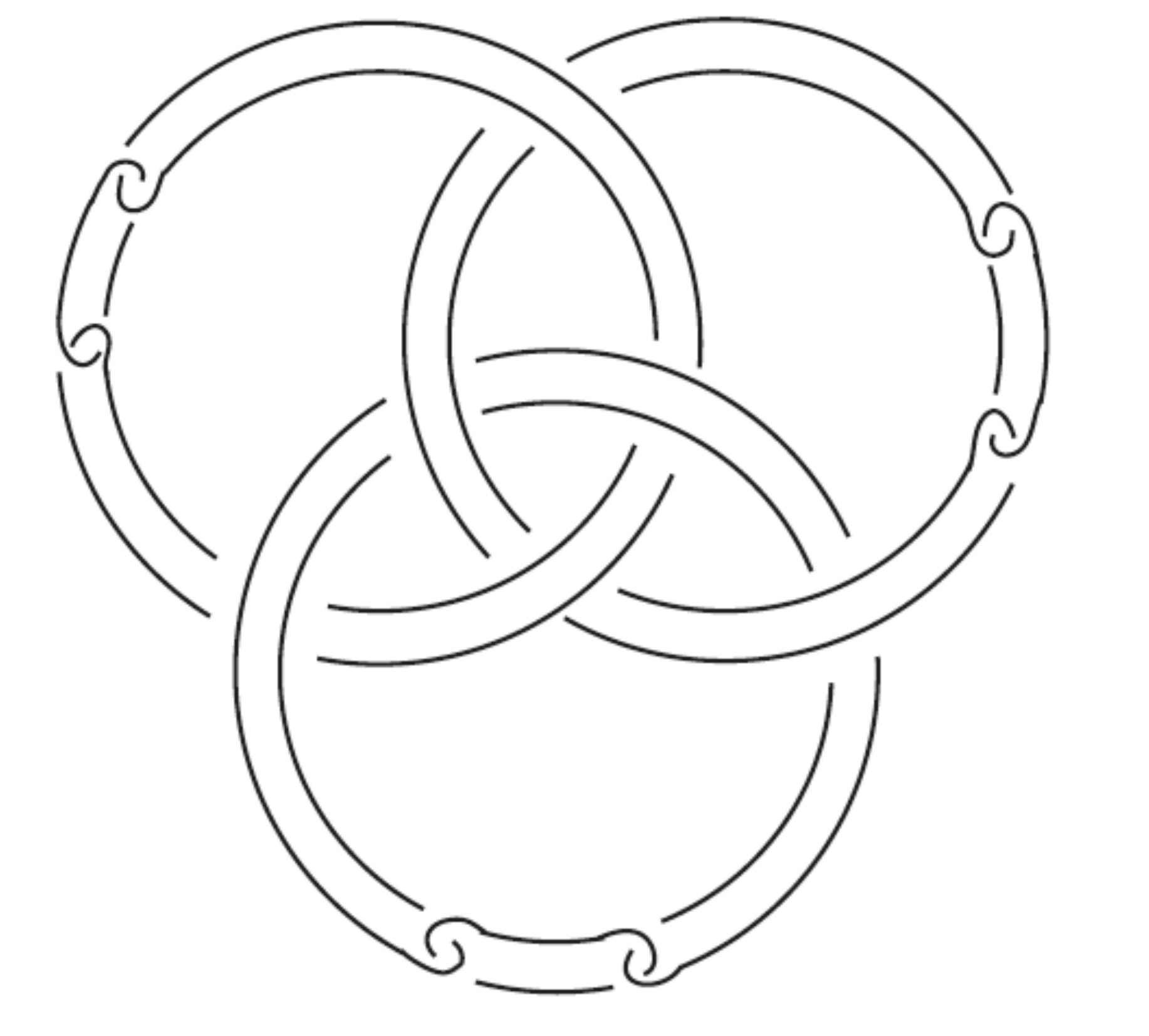}}\qquad
\subfigure[$BD(BD(BR))=BD_2(BR)$]{\label{fig:BDBDBR}\includegraphics[height=5.2cm]{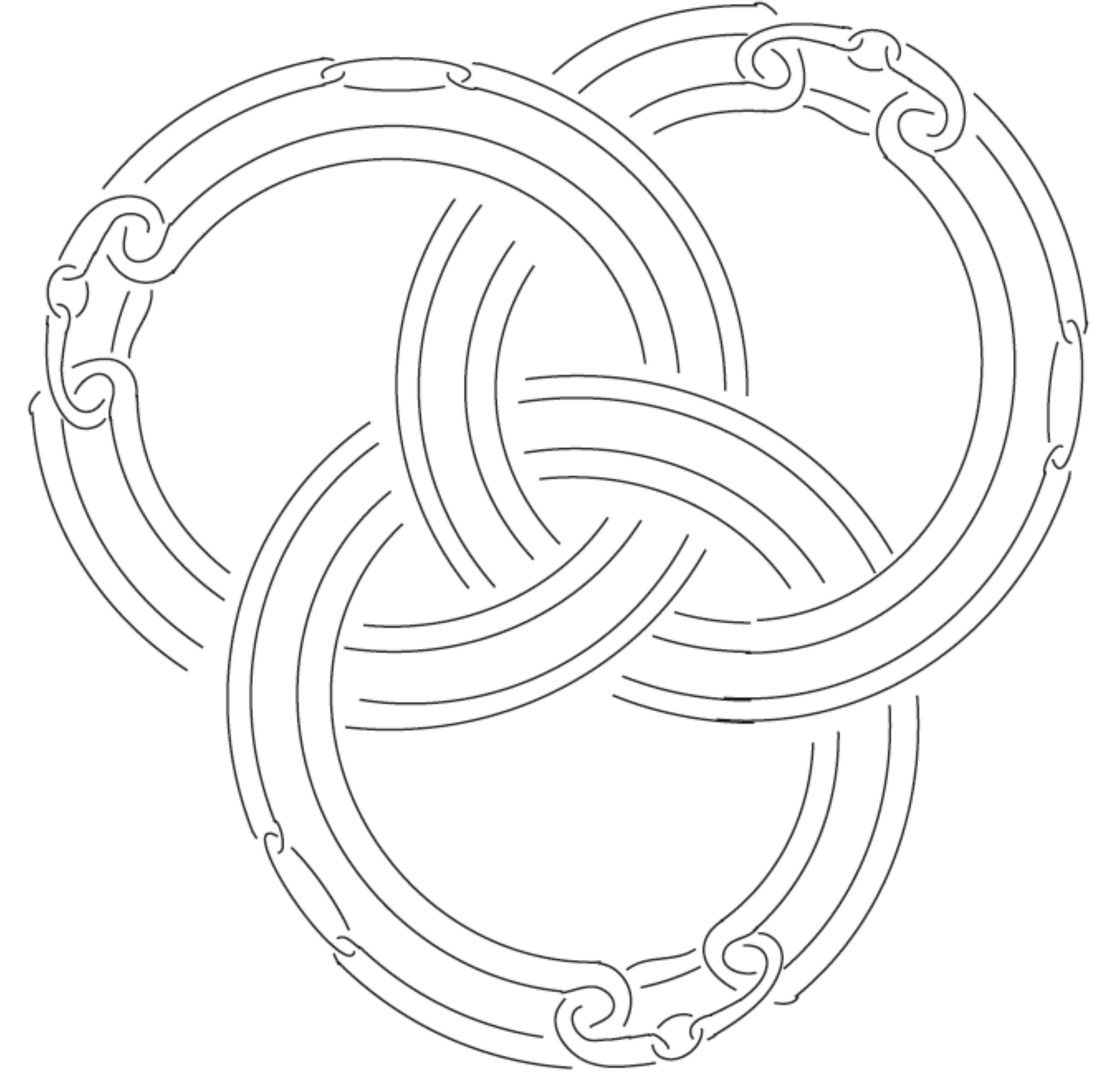}}
\caption[Bing doubles of the Borromean Rings]
{Examples of iterated Bing doubles of the Borromean Rings}
\end{figure}

The example exhibited in the above proof came from iterated Bing doubles of a link with certain nonvanishing $\bar \mu$-invariants.  These iterated Bing doubles will always have a nonzero $\bar \mu$-invariant~\cite{C4}, and the above example is also not concordant to a boundary link.  Hence our results are not concordant to those previously known.

Since the knot concordance group $\mathcal{C}$ is abelian, all successive quotients of the ($n$)-solvable filtration are abelian.  It is known, however, that $\mathcal{C}^m$ is a nonabelian group for $m\geq 2$~\cite{LD}.  We briefly recall some facts known about certain quotient groups of $\{\mathcal{F}^m_n\}$.

The quotient $\mathcal{F}^m_{-0.5}/\mathcal{F}_0^m$ has been classified by Martin and is known to be abelian ~\cite{TM}

We also know that the quotient $\mathcal{C}^m/\mathcal{F}^m_0$ is a nonabelian group for $m\geq 3$.

\begin{exam}
Consider the pure braids in Figure~\ref{fig:A} and~\ref{fig:B}.  We build the commutator $ABA^{-1}B^{-1}$ seen in Figure~\ref{fig:commutator}. The link $\widehat {ABA^{-1}B^{-1}}$ is isotopic to the Borromean Rings which are not (0)-solvable. We conclude that $\mathcal{C}^m/\mathcal{F}^m_0$ is not abelian.

\begin{figure}[h!]
\centering
\subfigure[$A$]{\label{fig:A}\includegraphics[height=2.5cm]{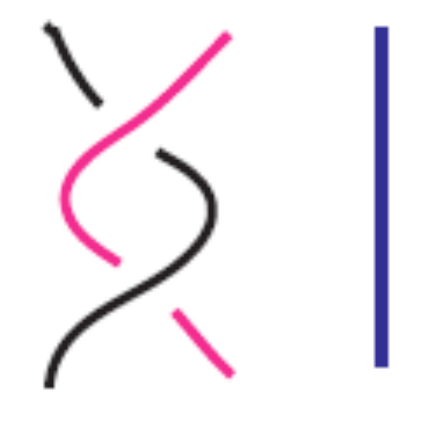}}\qquad \qquad
\subfigure[$B$]{\label{fig:B}\includegraphics[height=2.5cm]{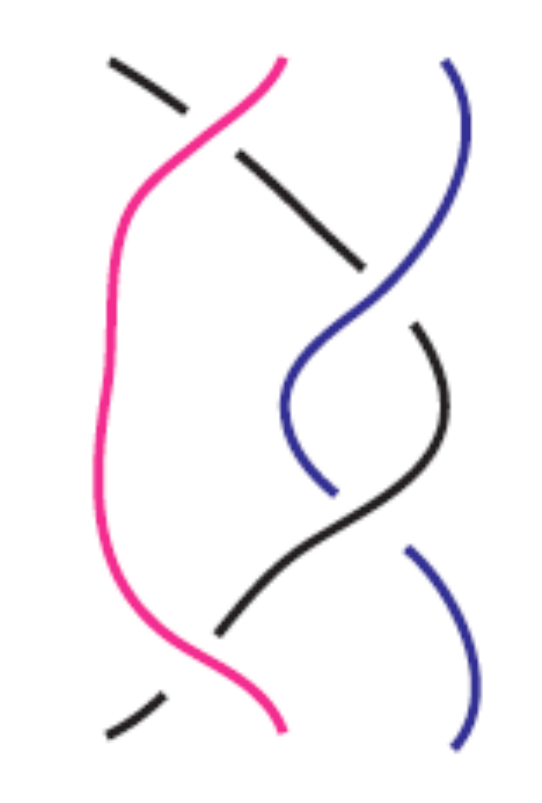}}\\
\subfigure[Pure braid $ABA^{-1}B^{-1}$]{\label{fig:commutator}\includegraphics[height=2cm]{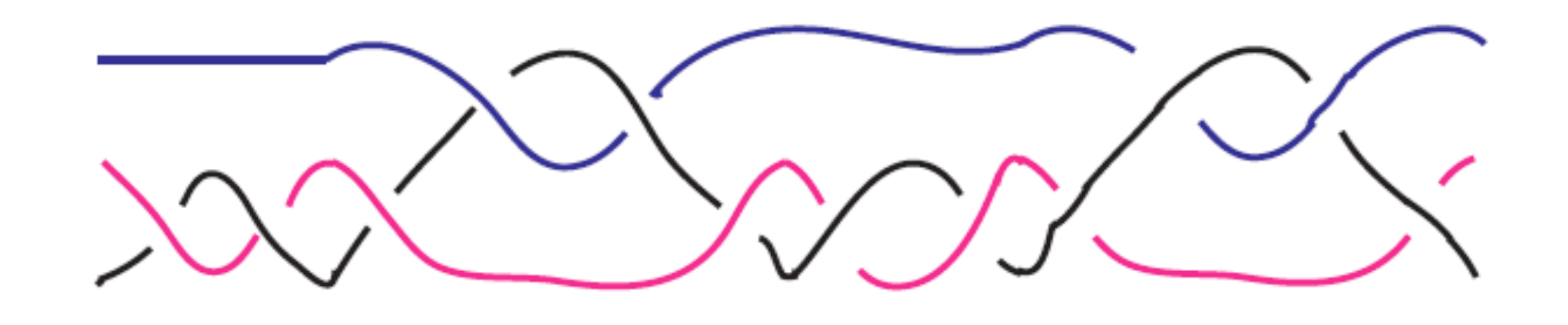}}
\caption[Example of a commutator of pure braids that is not (0)-solvable]
{Example of a commutator of pure braids that is not (0)-solvable}
\end{figure}

\end{exam}

We continue on with our investigation of quotients of $\{\mathcal{F}^m_n\}$.  Again using Theorem~\ref{theorem:main}, we will show that $\mathcal{F}^m_{-0.5}/\mathcal{F}^m_1$ is a nonabelian group.

\begin{theorem}
$\mathcal{F}^m_{-0.5}/\mathcal{F}^m_1$ is a nonabelian group for $m \geq 3$.
\label{theorem:abelian}\end{theorem}

In order to prove this theorem, we need to demonstrate that there exists two string links with pairwise linking numbers equal to zero such that when we construct the commutator we get a string link that is not (1)-solvable.

\begin{proof}
The Borromean Rings, $BR$, can be written as a pure braid, specifically, $BR=\sigma_2\sigma_1^{-1}\sigma_2\sigma_1^{-1}\sigma_2\sigma_1^{-1}$ (see Figure~\ref{fig:BRbraid}).  Consider the pure braid $\sigma_1BR\sigma_1^{-1}$, the Borromean Rings conjugated by $\sigma_1$ (see Figure~\ref{fig:combr}).  We look at the commutator $L=(BR)(\sigma_1BR\sigma_1^{-1})(BR)^{-1}(\sigma_1BR\sigma_1^{-1})^{-1}$.  Notice that $L$ is also a pure braid.

For braids, the canonical meridians, $m_i$, will freely generate the fundamental group and any other meridian of $L_i$ (the $i^{th}$ string of $L$) in $\pi_1$ will be a conjugate of $m_i$.  This allows us to write $l_i$ of $\hat L$ as a product of the $m_i$'s using an algorithmic procedure.
\begin{figure}[h!]
\centering
\subfigure[$BR$ as a pure braid]{\label{fig:BRbraid}\includegraphics[height=2cm]{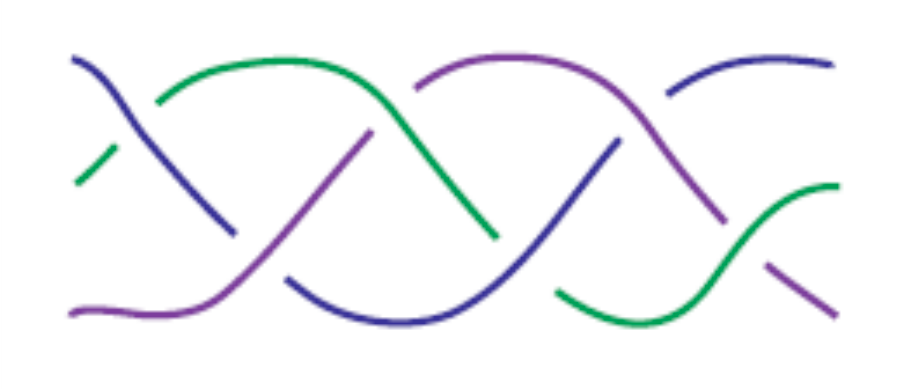}}\qquad \qquad \qquad
\subfigure[$\sigma_1BR\sigma_1^{-1}$]{\label{fig:combr}\includegraphics[height=2cm]{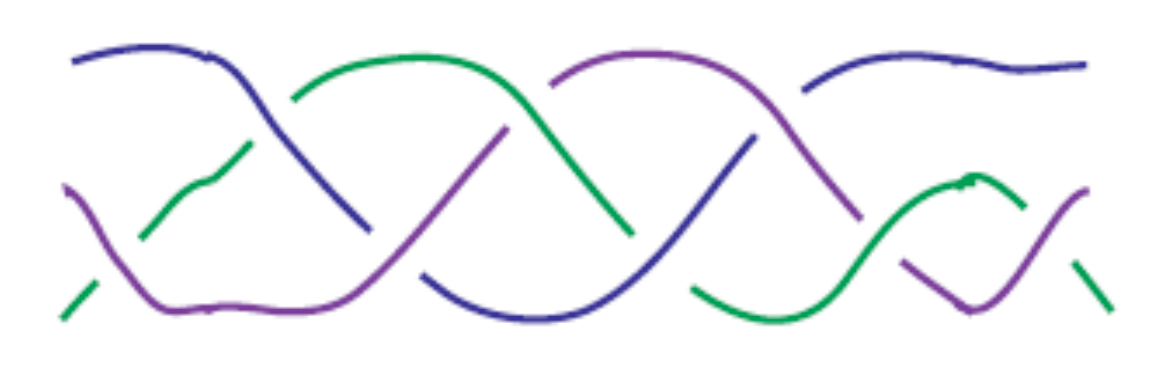}}
\caption[Borromean Rings as a pure braid and a conjugate of them]
{Borromean Rings as a pure braid and a conjugate of them}
\end{figure}
Using this idea, Davis designed a computer program to compute this invariants for braids~\cite{Dweb}. We found that $\bar \mu_L(313323)=-1$.  By Theorem~\ref{theorem:main}, $L$ is not (1)-solvable.  Therefore $\mathcal{F}^m_{-0.5}/\mathcal{F}^m_1$ is a nonabelian group.
\end{proof}

\section{The Grope Filtration and The ($n$)-Solvable Filtration}

In addition to defining the ($n$)-solvable filtration, Cochran, Orr, and Teichner~\cite{COT} also defined the Grope filtration, $\{\mathcal{G}_n^m\}$ of the (string) link concordance group,
$$\{0\}\subset \cdots \subset \mathcal{G}^m_{n+1} \subset \mathcal{G}^m_{n.5} \subset \mathcal{G}^m_n \subset \cdots \subset \mathcal{G}^m_{0.5} \subset \mathcal{G}^m_{0}\subset \mathcal{C}^m.$$

\begin{definition} A \textbf{grope} is a special pair (2-complex, base circle) which has a height $n \in \frac{1}{2}\mathbb{N}$ assigned to it.  A grope of height 1 is precisely a compact, oriented surface $\Sigma$ with a single boundary component, which is the base circle (see Figure~\ref{fig:gropes}).
\end{definition}

A grope of height $n+1$ can be defined recursively by the following construction.  Let $\{\alpha_i, \beta_i: i=1, \dots, 2g\}$ where $g$ is the genus of $\Sigma$, be a symplectic basis of curves for $H_1(\Sigma)$, where $\Sigma$ is a height one grope. The surface $\Sigma$ is also known as the first stage grope.  Then a grope of height $n+1$ is formed by attaching gropes of height $n$ to each $\alpha_i$ and $\beta_i$ along the base circles (see Figure~\ref{fig:gropes}).  A grope of height $1.5$ is a surface with surfaces attached to `half' of the basis curves.  A grope of height $n+1.5$ is obtained by gluing gropes of height $n$ to the $\alpha_i$ and gropes of height $n+1$ to the $\beta_i$.

\begin{figure}[h!]
  \centering
  \includegraphics[height=3cm]{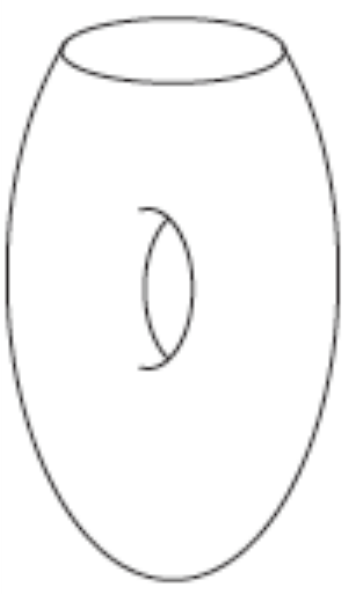} \qquad \includegraphics[height=3cm]{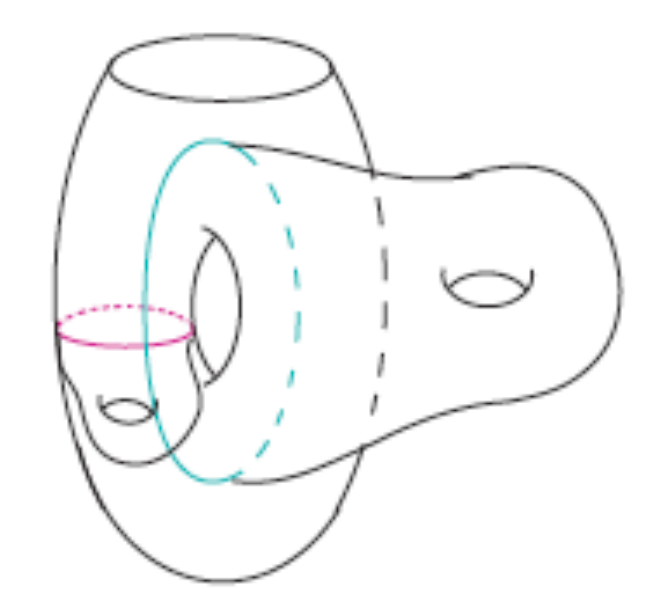}
  \caption[Examples of Gropes]%
  {A height 1 and height 2 grope}
  \label{fig:gropes}
\end{figure}

Given a 4-manifold, $W$, with boundary $S^3$ and a framed circle $\gamma \subset S^3$, we say that $\gamma$ bounds a \textbf{Grope} in $W$ if $\gamma$ extends to a smooth embedding of a grope with its untwisting framing (parallel push offs of Gropes can be taken in $W$).

We denoted $\mathcal{G}^m_n$ to be the subset of $\mathcal{C}^m$ defined by the following.  A string link $L$ is in $\mathcal{G}^m_n$ if the components of $\hat L$ bound disjoint Gropes of height $n$ in $D^4$.  It can be shown that these subsets are actually normal subgroups of $\mathcal{C}^m$.  Harvey showed that this filtration is nontrivial by looking at the filtration of boundary string links~\cite{CH2}.

There is also a notion of Grope concordance.  To define this, the following definition is needed.
\begin{definition} An \textbf{annular grope} of height $n$ is a grope of height $n$ that has an extra boundary component on its first stage.
\end{definition}

 The two boundary components of an annular grope are said to cobound an annular grope.  Two links, $L_0$ and $L_1$, are \emph{height $n$ Grope concordant} if their components cobound disjoint height $n$ annular Gropes, $G_i$, in $S^3 \times [0,1]$ such that $G_i \cap (S^3\times \{j\})=$ the $i^{th}$ component of $L_j$ where $j=0,1$.

 Thus far, two filtrations of the string link concordance group $\mathcal{C}^m$ have been defined.  The ($n$)-solvable filtration is an algebraic approximation while the Grope filtration is a geometric approximation to a link being slice.  It is a natural question to ask whether these two filtrations are related.  In order to answer this question, we need to analyze the relationship between a link bounding disjoint gropes and the link's Milnor's $\bar \mu$-invariants.

\begin{definition}
Let $L=L_1 \cup L_2 \cup \cdots \cup L_m$ and $L'=L'_1 \cup L'_2 \cup \cdots \cup L'_m$ be ordered, oriented links in $S^3$.  We say that $L$ is \textbf{$k$-cobordant} to $L'$, where $k\in \mathbb{Z^+}$, if there are disjoint, smoothly embedded compact, connected, oriented surfaces $V_1, V_2, \dots, V_m$ in $S^3 \times [0,1]$ with $\partial V_i = \partial_0V_i \cup \partial_1V_i$ such that for all $i=1, \dots, m$, we have
\begin{itemize}
\item[i.] $V_i \cap (S^3 \times \{0\}) = \partial_0 V_i=L_i$ and $V_i \cap (S^3 \times \{1\}) = \partial_1 V_i=L'_i$;
\item[ii.] there is a tubular neighborhood $V_i \times D^2$ of $V_i$ in $S^3\times[0,1]$ which extends the ``longitudinal" ones of $\partial V_i=L_i \cup L'_i$ in $S^3 \times \{0\}$ and $S^3 \times \{1\}$ resp such that the image of the homomorphism$$ \pi_1(V_i) \rightarrow \pi_1(V_i \times \partial D^2) \rightarrow \pi_1(S^3 \times [0,1] - V)=G$$ lies in the $k$th term of the lower central series of $G$, $G_k$.
\end{itemize}
\end{definition}

A link that is $k$-cobordant to a slice link is called \textbf{null $k$-cobordant}.

The concept of $k$-cobordism is related to the grope filtration as seen in the following proposition.
\begin{proposition}
If $L \in \mathcal{G}_{n}^m$, then it is $2^{n-1}$-cobordant to a slice link.
\end{proposition}

\begin{proof}
Suppose $L \in \mathcal{G}_{n+2}^m$.  Then the components of $L$, say $\ell_i$, bound disjoint Gropes of height $n$ in $D^4 \cong S^3 \times [0,1]$.  Moreover, the $\ell_i$s extend to smooth embeddings of gropes with their untwisting framing.  Also, $L$ is height $n$ Grope concordant to a slice link $L'$.  Let $V_i$ be the first stage Grope bounded by $\ell_i$ and $\ell_i'$ (ie. the annular Grope in the concordance). Let $V=\coprod_{i=1}^mV_i$.

Now consider the homomorphism
$$ \pi_1(V_i) \rightarrow \pi_1(V_i \times \partial D^2) \rightarrow \pi_1(S^3 \times [0,1] - V)=G$$
that is induced by pushing $V_i$ off itself in the normal direction.  Let $\{ \alpha_i, \beta_i\}$ be a sympletic basis for $V_i$ (see Figure~\ref{fig:gropebasis}).  The parallel push-offs of Gropes can be taken in $S^3 \times [0,1]$ and thus are now in $S^3 \times [0,1] - V$.  By the construction of the Gropes, each of the $\alpha_i$s and $\beta_i$s bound Gropes of height $n-1$ in the exterior of $V$.  Thus $$[\alpha_i],\text{ }[\beta_i] \in G^{(n-1)} \subset G_{2^{n-1}}$$ by the fact that if a curve $\ell$ bounds a (map of a) grope of height $n$ in a space $X$, then $[\ell] \in \pi_1(X)^{(n)}$.
This concludes the proof.
\end{proof}

\begin{figure}[h!]
\centering
\includegraphics[height=4.5cm]{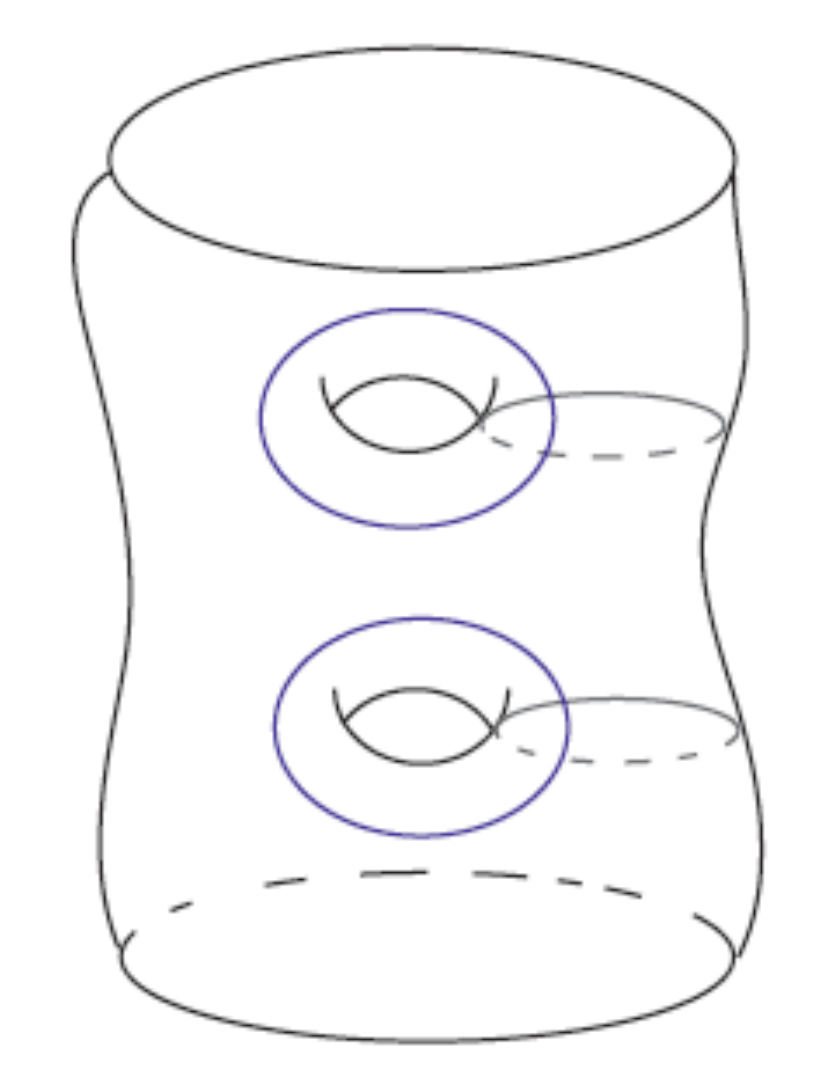}
\put(-78,75){$\beta_1$}
\put(-79,43){$\beta_2$}
\put(-5,73){$\alpha_1$}
\put(-6,40){$\alpha_2$}
\put(-95,110){$\ell_i$}
\put(-80,-3){$\ell'_i$}
\put(5,115){$V_i$}
\caption[First stage grope with basis]
{The first stage grope, $V_i$ with symplectic basis $\{\alpha_i, \beta_i\}_{i=1,2}.$}
\label{fig:gropebasis}
\end{figure}
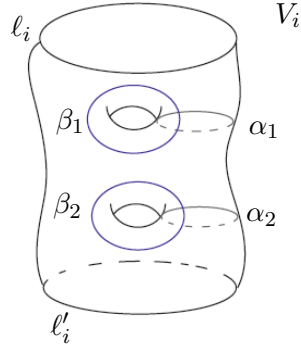

The following corollary of Lin~\cite{Lin} relates Milnor's invariants with $k$-cobordant links.

\begin{corollary}[Lin]
If $L$ and $L'$ are $k$-cobordant, then Milnor's $\bar \mu$-invariants of $L$ and $L'$ with lengths less than or equal to $2k$ are the same.  In particular, if $L$ is null $k$-cobordant, then $\bar \mu_L(I)=0$ for $|I| \leq 2k$.
\end{corollary}

\begin{corollary}A link $L$ with components that bound disjoint Gropes of height $n$ has $\bar \mu_L(I) = 0$ for $|I|\leq 2^{n}$.
\label{cor:grope}\end{corollary}

\begin{proof}
The proof of this is immediate from the previous two results.
\end{proof}

Cochran, Orr and Teichner~\cite{COT} showed that these two filtrations are related.
\begin{theorem}[Cochran-Orr-Teichner]
If a link $L$ bounds a grope of height $n+2$ in $D^4$, then $L$ is $(n)$-solvable, i.e. $\mathcal{G}^m_{n+2} \subseteq \mathcal{F}^m_n$ for all $m$ and $n$.
\end{theorem}
The natural question is whether or not the inclusion goes in the other direction.  Recall from Theorem~\ref{theorem:main} that an ($n$)-solvable string link has vanishing $\bar \mu$-invariants for lengths less than or equal to $2^{n+2}-1$, whereas in Corollary~\ref{cor:grope} a string link in $\mathcal{G}_{n+2}^m$ has vanishing $\bar \mu$-invariants for lengths less than or equal to $2^{n+2}$.  This difference of one gives motivation to try to find a nontrivial element in $\mathcal{F}^m_n/\mathcal{G}^m_{n+2}$.

\begin{corollary}
$\mathcal{F}^m_n/\mathcal{G}^m_{n+2}$ is nontrivial for $m \geq 2^{n+2}$.  Moreover, $\mathbb{Z} \subset \mathcal{F}^m_n/\mathcal{G}^m_{n+2}$ in this case.
\label{cor:main}\end{corollary}

\begin{proof}
Let $H$ be the Hopf link.  By Proposition~\ref{proposition:bingdouble}, $BD(H) \in \mathcal{F}_0$, where $BD(H)$ is the Bing double $H$. The invariant $\bar \mu_H(12)= \pm 1$ depending on orientation, as it is just the linking number between the two components.  Again, by work Cochran given in Chapter 8 of~\cite{C4}, $\bar \mu_{BD(H)}(I)= \pm 1$ for some $I$ of length 4.  Using iterated Bing doubling we achieve $BD_{n+1}(H)$ in $\mathcal{F}_n$ by Proposition~\ref{proposition:bingdouble}, and $\bar \mu_{BD_{n+1}(H)}(I) = \pm 1$ for some $I$ of length $2^{n+2}$. $BD_{n+1}(H)$ is ($n$)-solvable, but since some $\bar \mu_{BD_{n+1}(H)}$ does not vanish for a length of $2^{n+2}$ it cannot bound a Grope of height $n+2$.

To show that there is an infinite cyclic subgroup contained within this quotient, we look at string link representatives of $H$ and $BD_{n+1}(H)$.  The proof of this result is completely analogous to the proof of Theorem~\ref{theorem:nontrivial}.
\end{proof}

  There are an infinite number of links that are ($n$)-solvable, but the components do not bound gropes of height $n+2$.




\bibliographystyle{amsalpha}
\bibliography{cottobib}

\providecommand{\bysame}{\leavevmode\hbox to3em{\hrulefill}\thinspace}
\providecommand{\MR}{\relax\ifhmode\unskip\space\fi MR }
\providecommand{\MRhref}[2]{%
  \href{http://www.ams.org/mathscinet-getitem?mr=#1}{#2}
}
\providecommand{\href}[2]{#2}
\begin{thebibliography}{CGO01}

\bibitem[Cas75]{CA}
A.~J. Casson, \emph{Link cobordism and {M}ilnor's invariant}, Bull. London
  Math. Soc. \textbf{7} (1975), 39--40. \MR{MR0362286 (50 \#14728)}

\bibitem[CGO01]{CGO}
Tim~D. Cochran, Amir Gerges, and Kent Orr, \emph{Dehn surgery equivalence
  relations on 3-manifolds}, Math. Proc. Cambridge Philos. Soc. \textbf{131}
  (2001), no.~1, 97--127. \MR{MR1833077 (2002c:57033)}

\bibitem[CH08]{CH2}
Tim~D. Cochran and Shelly Harvey, \emph{Homology and derived series of groups.
  {II}. {D}wyer's theorem}, Geom. Topol. \textbf{12} (2008), no.~1, 199--232.
  \MR{MR2377249}

\bibitem[CH10]{CH4}
\bysame, \emph{Homological stability of series of groups}, Pacific Math. J.
  \textbf{246} (2010), no.~1, 31--47.

\bibitem[Cha10]{Cha4}
Jae~Choon Cha, \emph{Link concordance, homology cobordism, and
  {H}irzebruch-type intersection form defects from iterated $p$-covers},
  Journal of the European Mathematical Society \textbf{12} (2010), 555--610.

\bibitem[CHH12]{CHH}
Tim~D. Cochran, Shelly Harvey, and Peter Horn, \emph{Filtering smooth
  concordance classes of topologically slice knots}, Preprint available at
  hhttp://arxiv.org/abs/1201.6283.

\bibitem[CHL09]{CHL3}
Tim~D. Cochran, Shelly Harvey, and Constance Leidy, \emph{Knot concordance and
  higher-order {B}lanchfield duality}, Geom. Topol. \textbf{13} (2009),
  1419--1482.

\bibitem[Coc90]{C4}
Tim~D. Cochran, \emph{Derivatives of links: {M}ilnor's concordance invariants
  and {M}assey's products}, Mem. Amer. Math. Soc. \textbf{84} (1990), no.~427,
  x+73. \MR{MR1042041 (91c:57005)}

\bibitem[COT03]{COT}
Tim~D. Cochran, Kent~E. Orr, and Peter Teichner, \emph{Knot concordance,
  {W}hitney towers and {$L\sp 2$}-signatures}, Ann. of Math. (2) \textbf{157}
  (2003), no.~2, 433--519. \MR{MR1973052 (2004i:57003)}

\bibitem[Dav11]{Dweb}
Christopher Davis, \emph{Christopher davis's webpage}, {http://math.rice.edu/
  {\textasciitilde} cwd1}, February 2011.

\bibitem[Dwy75]{Dw}
William~G. Dwyer, \emph{Homology, {M}assey products and maps between groups},
  J. Pure Appl. Algebra \textbf{6} (1975), no.~2, 177--190. \MR{MR0385851 (52
  \#6710)}

\bibitem[Har08]{Ha2}
Shelly~L. Harvey, \emph{Homology cobordism invariants and the
  {C}ochran-{O}rr-{T}eichner filtration of the link concordance group}, Geom.
  Topol. \textbf{12} (2008), no.~1, 387--430. \MR{MR2390349}

\bibitem[HL90]{HL1}
Nathan Habegger and Xiao-Song Lin, \emph{The classification of links up to
  link-homotopy}, J. Amer. Math. Soc. \textbf{3} (1990), no.~2, 389--419.
  \MR{MR1026062 (91e:57015)}

\bibitem[IO01]{IO}
Kiyoshi Igusa and Kent~E. Orr, \emph{Links, pictures and the homology of
  nilpotent groups}, Topology \textbf{40} (2001), no.~6, 1125 -- 1166.

\bibitem[LD88]{LD}
Jean-Yves Le~Dimet, \emph{Cobordisme d'enlacements de disques}, M\'em. Soc.
  Math. France (N.S.) (1988), no.~32, ii+92. \MR{MR971415 (90e:57046)}

\bibitem[Lin91]{Lin}
Xiao-Song Lin, \emph{Null k-cobordant links in ${S}^3$}, Comm. Math. Helv.
  \textbf{66} (1991), no.~3, pp. 333--339.

\bibitem[Mar]{TM}
Taylor Martin, \emph{Lower order solvability of links}, In preparation.

\bibitem[Mat87]{MA}
S.~V. Matveev, \emph{Generalized surgery of three-dimensional manifolds and
  representations of homology spheres}, Mathematical Notes \textbf{42} (1987),
  651--656, 10.1007/BF01240455.

\bibitem[Mil54]{M1}
John Milnor, \emph{Link groups}, Ann. of Math. (2) \textbf{59} (1954),
  177--195. \MR{MR0071020 (17,70e)}

\bibitem[Mil57]{M2}
\bysame, \emph{Isotopy of links. {A}lgebraic geometry and topology}, A
  symposium in honor of S. Lefschetz, Princeton University Press, Princeton, N.
  J., 1957, pp.~280--306.

\bibitem[MN89]{MN}
Hitoshi Murakami and Yasutaka Nakanishi, \emph{On a certain move generating
  link-homology}, Mathematische Annalen \textbf{284} (1989), 75--89,
  10.1007/BF01443506.

\bibitem[NS03]{NS}
S.~Naik and T.~Stanford, \emph{A move on diagrams that generates s-equivalence
  of knots}, Jour. of Knot Theory and its Ramifications \textbf{12-5} (2003),
  717--714.

\bibitem[Orr89]{O1}
Kent~E. Orr, \emph{Homotopy invariants of links}, Invent. Math. \textbf{95}
  (1989), no.~2, 379--394. \MR{MR974908 (90a:57012)}

\bibitem[Ott11]{OT}
Carolyn Otto, \emph{The ($n$)-solvable filtration of the link concordance group
  and milnor's $\bar \mu$-invariants}, PhD. Thesis.

\end{thebibliography}
\end{document}